\documentclass[11pt]{article}
\usepackage[utf8]{inputenc}

\usepackage{amsmath, amssymb, amsthm}
\usepackage[none]{hyphenat}

\usepackage{tikz}
\usetikzlibrary{calc}
\usetikzlibrary{math}
\usetikzlibrary{decorations,decorations.markings,decorations.pathreplacing,decorations.pathmorphing}
\usetikzlibrary{fadings}
\usetikzlibrary{fit}
\usetikzlibrary{matrix}
\usetikzlibrary{patterns}
\usetikzlibrary{positioning}
\usetikzlibrary{shadows}
\usetikzlibrary{shapes,shapes.geometric}
\usetikzlibrary{tikzmark}
\usetikzlibrary{trees}
\usetikzlibrary{overlay-beamer-styles}
\usepackage{tikzsymbols}

\allowdisplaybreaks
\expandafter\let\expandafter\oldproof\csname\string\proof\endcsname
\let\oldendproof\endproof
\renewenvironment{proof}[1][\proofname]{%
	\oldproof[\bf #1]%
}{\oldendproof}

\allowdisplaybreaks

\theoremstyle{plain}
\newtheorem{theorem}{Theorem}
\newtheorem{lemma}{Lemma}[section]

\newtheorem{conjecture}[theorem]{Conjecture}

\title{Polynomial removal lemmas for ordered graphs}
\author{Lior Gishboliner\thanks{ETH Zurich, \emph{e-mail}: \textbf{\{lior.gishboliner,istvan.tomon\}@math.ethz.ch}}
	\and
	Istv\'an Tomon\footnotemark[1] }
\date{}

\newcommand{\core}{\text{core}}
\newcommand{\poly}{\text{poly}}

\begin{document}
\sloppy 
\maketitle

\begin{abstract}
    A recent result of Alon, Ben-Eliezer and Fischer establishes an induced removal lemma for ordered graphs. That is, if $F$ is an ordered graph and $\varepsilon>0$, then there exists $\delta_{F}(\varepsilon)>0$ such that every $n$-vertex ordered graph $G$ containing at most $\delta_{F}(\varepsilon) n^{v(F)}$ induced copies of $F$ can be made induced $F$-free by adding/deleting at most $\varepsilon n^2$ edges.  We prove that $\delta_{F}(\varepsilon)$ can be chosen to be a polynomial function of $\varepsilon$ if and only if $|V(F)|=2$, or $F$ is the ordered graph with vertices $x<y<z$ and edges $\{x,y\},\{x,z\}$ (up to complementation and reversing the vertex order). We also discuss similar problems in the non-induced case.
\end{abstract}

\section{Introduction}

Graph removal lemmas are among the most powerful tools in combinatorics, with further applications in number theory, logic, and property testing. The celebrated graph removal lemma, which originates in the work of Ruzsa and Szemer\'edi \cite{RuzsaSz}, states that if $F$ is a graph and $\varepsilon>0$, then there exists $\delta=\delta_{F}(\varepsilon)>0$ such that every $n$-vertex graph $G$ containing at most $\delta n^{v(F)}$ copies of $F$ can be made $F$-free by deleting at most $\varepsilon n^2$ edges. Alon, Fischer, Krivelevich and Szegedy \cite{AFKS} established an analogue of this for induced subgraphs. This result, known as the induced removal lemma, states that if $G$ contains at most $\delta_F(\varepsilon) n^{v(F)}$ induced copies of a graph $F$, then $G$ can be made induced $F$-free by adding/deleting at most $\varepsilon n^2$ edges. A generalization to arbitrary hereditary graph properties was later obtained by Alon and Shapira \cite{AS_hereditary}. For a general survey on graph removal lemmas, we refer the reader to \cite{CF}. 

In this paper, we are interested in ordered variants of the graph removal lemma. An \emph{ordered graph} is a graph with a linear ordering $\leq$  on its vertex set. An ordered graph $H$ is a subgraph of an ordered graph $G$ if there exists an order preserving embedding from $V(H)$ to $V(G)$ which maps edges into edges, and an induced subgraph if it also maps non-edges into non-edges. The natural analogue of the induced (and also non-induced) removal lemma for ordered graphs was established by Alon, Ben-Eliezer and Fischer \cite{ABF17}; see also \cite{AB20} for related results.

	All of the above results build on the regularity lemma of Szemer\'edi \cite{Szemeredi} or its appropriate generalizations. Consequently, the lower bounds on $\delta_{F}(\varepsilon)$ supplied by these proofs are quite poor. Even in the case of the original graph removal lemma, the current best known bound is $1/\delta \leq \text{tower}(O(\log 1/\varepsilon))$, as proved by Fox in \cite{Fox}. Here, $\text{tower}(x)$ denotes a tower of $x$ exponents. On the other hand, in some special cases better bounds are known. This motivated the natural question of characterizing the cases in which the removal lemma has polynomial bounds, namely, when $1/\delta$ can be taken as a polynomial function of $1/\varepsilon$. By now there are several results of this type. In the case of graphs, Alon \cite{Alon} showed that the $F$-removal lemma has polynomial bounds if and only if $F$ is bipartite. For the case of induced subgraphs, a result of Alon and Shapira \cite{AS_induced} tells us that the induced $F$-removal lemma does not have polynomial bounds, unless $|V(F)|=2$, or $F\in \{P_3,\overline{P_3},P_4,C_{4},\overline{C_4}\}$, where $P_{k},C_{k}$ are the path and cycle on $k$ vertices, respectively. A polynomial bound in the case $F=P_3$ is easy to show. Alon and Fox \cite{AF} proved that the induced-$P_4$-removal lemma has polynomial bounds as well. The case $F=C_4$ is still open, see \cite{GS_C4,GS_poly} for the currently best known bounds. The results of \cite{AS_hypergraphs} and \cite{GT} completely characterize the $k$-uniform hypergraphs which admit polynomial induced removal lemmas, for $k\geq 3$.

\subsection{Polynomial induced removal lemma for ordered graphs}

  In the extended version of \cite{ABF17}, Alon, Ben-Eliezer and Fischer proposed the problem of finding ordered graph properties with  polynomial induced removal lemmas. Addressing this question, we give a complete characterization of ordered graphs $F$ for which the induced-$F$-removal lemma has polynomial bounds. It turns out that there is essentially only one such nontrivial ordered graph. For an ordered graph $G$, we denote by $\overline{G}$ the complement of $G$, and by $G^{\leftarrow}$ the ordered graph obtained by reversing the vertex order. It is easy to see that the (induced/non-induced) removal lemma for $F$ is equivalent to the (induced/non-induced) removal lemma for $F^{\leftarrow}$. In the induced case, there is also symmetry with respect to complementation: the induced removal lemma for $F$ is equivalent to the induced removal lemma for $\overline{F}$. In the rest of the paper we will denote by $D$ the ordered graph with vertices $x < y < z$ and edges $\{x,y\},\{x,z\}$. 

\begin{theorem}\label{thm:induced}
	For an ordered graph $F$, the induced $F$-removal lemma has polynomial bounds if and only if $|V(F)|=2$, or  $F \in \left\{D,D^{\leftarrow},\overline{D},\overline{D^{\leftarrow}}\right\}$. 
\end{theorem}
 
 A graph is \emph{chordal} if it contains no induced cycle of length at least 4. It is well known that a graph is chordal if and only if it has a vertex ordering such that the resulting ordered graph is induced $D$-free (such an ordering is called a perfect elimination order), see e.g. \cite{FG}. It was recently proved by de Joannis de Verclos \cite{J} that the property of being chordal also admits a polynomial removal lemma. Despite the similarity with Theorem~\ref{thm:induced}, it is unclear whether there are any implications between these two results, as the ordering imposes additional structure.

\subsection{Polynomial (non-induced) removal lemma for ordered graphs}
We also consider the non-induced variant of the previous theorem. In the case of \emph{directed graphs} (\emph{digraphs}), Alon and Shapira  \cite{AS_directed} gave a complete characterization of digraphs $F$ such that the $F$-removal lemma has polynomial bounds. This result can be stated as follows. A \emph{homomorphism} from a digraph $G_1$ to a digraph $G_2$ is a function $\varphi:V(G_1)\rightarrow V(G_2)$ that satisfies $(u,v)\in E(G_1) \Rightarrow (\varphi(u),\varphi(v))\in E(G_2)$. The \emph{core} of a digraph $G$ is the smallest subgraph $K$ for which there is a homomorphism from $G$ to $K$. Then, for a connected digraph $F$, the $F$-removal lemma has polynomial bounds if and only if the core of $F$ is an oriented tree or a directed cycle of length 2.

In a highly parallel manner, we propose a conjecture characterizing ordered graphs which admit a polynomial removal lemma, and prove its ``only if'' part. A \emph{homomorphism} between ordered graphs is a graph homomorphism which also preserves the vertex orderings. Formally, for two ordered graphs $G_1,G_2$, a map $\varphi:V(G_1) \rightarrow V(G_2)$ is a homomorphism if $\{\varphi(x),\varphi(y)\} \in E(G_2)$ for every $\{x,y\} \in E(G_1)$, and $\varphi(x) \leq \varphi(y)$ for every  $x,y \in V(G_1)$ satisfying $x\leq y$. 
Observe that for such $\varphi$, the preimage of each vertex of $G_2$ is an interval in $G_1$ which spans an independent set. 
%When saying ``homomorphism'' we always mean homomorphisms of ordered graphs.
For an ordered graph $G$, the {\em core} of $G$, denoted by $\core(G)$, is the smallest subgraph of $G$ (in terms of number of vertices) to which there is a homomorphism from $G$. In the preliminaries, we will show that the core is well defined, that is, the smallest such subgraph is unique up to isomorphism. We prove that if $\core(F)$ is not a forest, then $F$ has no polynomial removal lemma. 

\begin{theorem}\label{thm:noninduced}
	Let $F$ be an ordered graph such that $\core(F)$ is not a forest, let $\varepsilon > 0$ be sufficiently small, and let $n \geq n_0(\varepsilon)$. Then there exists an ordered graph $G$ on $n$ vertices such that $G$ contains at most $\varepsilon^{\Omega(\log 1/\varepsilon)}n^{v(F)}$ copies of $F$, but one has to remove at least $\varepsilon n^{2}$ edges to destroy all copies of $F$ in $G$.
\end{theorem}

Unfortunately, we were unable to prove that the converse also holds in general, and leave it as an interesting open problem.

\begin{conjecture}\label{conj:noninduced}
	For an ordered graph $F$, if $\core(F)$ is a forest, then the $F$-removal lemma has polynomial bounds.
\end{conjecture}

In order to prove Conjecture \ref{conj:noninduced}, it is enough to show that it holds when $F$ itself is a forest. Indeed, let $K = \core(F)$ and suppose that the $K$-removal lemma has polynomial bounds. Let us assume that $V(K) = [k]$ and that the vertex order on $V(K)$ is given by the natural order on $[k]$. Let $s_i$ be the number of vertices of $F$ mapped to $i \in [k]$ under some fixed homomorphism $\varphi : F \rightarrow K$. 
Let $G$ be an ordered graph which is $\varepsilon$-far from being $F$-free. Then $G$ is also $\varepsilon$-far from being $K$-free (because $K$ is a subgraph of $F$), and hence $G$ contains at least $\delta n^{k}$ copies of $K$, where $\delta = \delta_K(\varepsilon) = \poly(\varepsilon)$. Consider the $k$-uniform hypergraph on $V(G)$ whose edges correspond to copies of $K$. It is easy to show, using a standard K\H{o}v\'ari-S\'os-Tur\'an-type argument, that this hypergraph contains at least $(1 - o(1))\delta^{s_1 \cdots s_k}n^{s_1 + \dots + s_k} = \poly(\varepsilon)n^{v(F)}$ copies of the complete $k$-uniform hypergraph $K^{(k)}_{s_1,\dots,s_k}$ in which the side of size $s_i$ appears before the side of size $s_j$ for every $1 \leq i < j \leq k$. Every such copy of $K^{(k)}_{s_1,\dots,s_k}$ contains a copy of $F$ in $G$. 

We remark that ordered forests are quite hard to analyze in many contexts, so it is not surprising that the corresponding question about removal lemmas is also difficult. For example, the extremal numbers (also known as Tur\'an numbers) of ordered forests are already not understood. This problem is the subject of the celebrated F\"uredi-Hajnal conjecture \cite{FH92}, see \cite{KTTW19} for the state of the art. 

Let us note that the results of this paper can also be stated in the language of {\em property testing}. A {\em tester} for a graph property $\mathcal{P}$ is a randomized algorithm which, given an input graph $G$ and an approximation parameter $\varepsilon$, distinguishes between the case that $G$ satisfies $\mathcal{P}$ and the case that $G$ is $\varepsilon$-far from $\mathcal{P}$, with success probability at least $\frac{2}{3}$ in both cases. Here, $G$ being {\em $\varepsilon$-far} from $\mathcal{P}$ means that one must add/delete at least $\varepsilon n^2$ edges to turn $G$ into a graph satisfying $\mathcal{P}$, where $n = |V(G)|$. The algorithm works by sampling vertices of $G$ and making edge-queries on pairs of sampled vertices. The measure of complexity, called {\em query complexity}, is the total number of queries that the algorithm makes. 
It turns out that many properties can be tested with query complexity which depends only on $\varepsilon$, i.e., is independent of the size of the input graph. This is in particular true for every hereditary property, as proved in \cite{AS_hereditary} for (unordered) graphs, and in \cite{ABF17} for ordered graphs.
A tester has {\em one-sided error} if it outputs the correct answer with probability $1$ in the case that $G$ satisfies $\mathcal{P}$. It is not hard to see that the optimal query complexity of a one-sided-error tester for (induced) $F$-freeness is essentially given by the bounds for the (induced) $F$-removal lemma. Hence, Theorem \ref{thm:induced} implies that in ordered graphs, induced $F$-freeness can be tested with query complexity $\poly(1/\varepsilon)$ with one-sided error if and only if $|V(F)| = 2$ or $F \in \left\{D,D^{\leftarrow},\overline{D},\overline{D^{\leftarrow}}\right\}$. For more on property testing, we refer the reader to the book of Goldreich \cite{Goldreich}.

\subsection{Preliminaries}
Given an (ordered) graph $G$, we denote by $v(G)$ its number of vertices, and by $e(G)$ its number of edges. Also, if $A,B\subseteq V(G)$ are disjoint, then $E_{G}(A,B)=E(A,B)$ is the set of edges between $A$ and $B$, and $e_{G}(A,B)=e(A,B)=|E(A,B)|$. Moreover, $\bar{E}(G)$ is the set of non-edges of $G$, $\bar{e}(G)=|\bar{E}(G)|$, $\bar{E}_{G}(A,B)=\bar{E}(A,B)$ is the set of non-edges between $A$ and $B$, and $\bar{e}_{G}(A,B)=\bar{e}(A,B)=|\bar{E}(A,B)| = |A||B| - e(A,B)$.

For a vertex $v$ and a set $X$, we denote by $N_X(v)$ the neighbourhood of $v$ inside $X$. We denote by $d(X) = e(X)/\binom{|X|}{2}$ the {\em density} of $X$, where $e(X)$ is the number of edges inside $X$. We say that a graph is a {\em disjoint union of cliques} if its vertex set partitions into cliques with no edges between them (equivalently, if the graph has no induced path with three vertices).

Given a set $X$ and a linear ordering $\leq$ on $X$, an \emph{interval} in $X$ is a set of the form $\{x\in X: a\leq x\leq b\}$ for some $a,b\in X$. We say that an (ordered) graph $G$ on $n$ vertices is \emph{$\varepsilon$-far} from an (ordered) graph property $\mathcal{P}$, if one has to add/delete at least $\varepsilon n^{2}$ edges in $G$ in order to turn it into a graph which has property $\mathcal{P}$. 

As promised in the introduction, let us show that $\core(G)$ is well defined. Let $K$ be a smallest subgraph of $G$ to which there is a homomorphism from $G$. Observe that every homomorphism $\varphi$ from $K$ to itself is surjective and hence bijective; indeed, otherwise one could compose $\varphi$ with a homomorphism from $G$ to $K$ to obtain a homomorphism from $G$ to a proper subgraph of $K$, in contradiction to the minimality of $K$. It follows that every homomorphism from $K$ to itself is an isomorphism. Since a homomorphism must also preserve the vertex order, it follows that the only homomorphism from $K$ to itself is the identity map. 
An ordered graph $K$ with the property that every homomorphism from $K$ to itself is the identity will be called a {\em core}. Note that $K$ is a core if and only if $\core(K) = K$. 

If $K_1,K_2$ are two smallest subgraphs of $G$ to which there are homomorphisms $\varphi_i : G \rightarrow K_i$, then $\varphi_1|_{V(K_2)}$ and $\varphi_2|_{V(K_1)}$ are both surjective, which implies that $K_1$ and $K_2$ are isomorphic.  

\section{Polynomial bounds for the induced-D-removal lemma}

In this section we prove the positive direction of Theorem \ref{thm:induced}. By symmetry with respect to complementation and order reversal, it is enough to prove the following:

\begin{theorem}\label{thm:P3}
	There exists $c>0$ such that the following holds. Let $G$ be an $n$-vertex ordered graph which is $\varepsilon$-far from being induced $D$-free. Then $G$ contains at least $\Omega(\varepsilon^{c}) \cdot n^3$ induced copies of $D$.
\end{theorem}

We prepare the proof with a few lemmas. One of our key lemmas shows that an ordered graph with few induced copies of $D$ can be partitioned into a constant number of almost-cliques and one additional set which contains few edges. In order to prove this, we first show that if $G$ is a dense ordered graph containing few copies of $D$, then $G$ contains a very dense subset of linear size.

\begin{lemma}\label{lem:extract_clique}
 	Let $G$ be an $n$-vertex ordered graph and let $\gamma,\delta > 0$. If $e(G) \geq \gamma n^2$ and $G$ contains at most $\delta \gamma^3 n^3/32$ induced copies of $D$, then there exists $X \subseteq V(G)$ with $|X| \geq \gamma n/2$ and $d(X) \geq 1 - \delta$.
 \end{lemma}
\begin{proof}
	For each $v \in V(G)$, let $N_v$ be the forward neighbourhood of $v$, namely, the set of vertices $u > v$ such that $\{u,v\} \in E(G)$. Since $\sum_{v \in V(G)}|N_v| = e(G) \geq \gamma n^2$, there is a set $U$ of at least $\gamma n/2$ vertices $v$ with $|N_v| \geq \gamma n/2$. If there exists $v\in U$ such that $d(N_v) \geq 1-\delta$, then we are done. Otherwise, each $v\in U$ is the first vertex in at least $\delta \binom{|N_v|}{2} \geq \delta \binom{\gamma n/2}{2} \geq \delta \gamma^2 n^2/16$ induced copies of $D$. Altogether, this gives 
	$\gamma n/2 \cdot \delta \gamma^2 n^2/16 = \delta \gamma^3 n^3/32$ induced copies of $D$, a contradiction. 
\end{proof}
\begin{lemma}\label{lem:partition}
	Let $G$ be an $n$-vertex ordered graph and let $\gamma,\delta > 0$. If $G$ contains at most $\delta \gamma^3 n^3/32$ induced copies of $D$, then there is a partition $V(G) = X_1 \cup \dots \cup X_m \cup Y$ with the following properties:
	\begin{enumerate}
		\item $|X_i| \geq \gamma n/2$ and $d(X_i) \geq 1 - \delta$ for $i = 1,\dots,m$.
		\item $e(Y) \leq \gamma n^2$.
	\end{enumerate}
\end{lemma}
\begin{proof}
	Set $V_0 = V(G)$. For $i \geq 0$, if $e(V_i) \leq \gamma n^2$ then set $Y = V_i$ and stop. Otherwise, set $\beta := \gamma n^2/|V_i|^2$ and
	apply Lemma \ref{lem:extract_clique} to $G[V_i]$ with parameters $\beta$ (instead of $\gamma$) and $\delta$. Note that $\delta \beta^3 |V_i|^3/32 = \delta\gamma^3 n^6/(32|V_i|^3) \geq
	\delta\gamma^3 n^3/32$. Hence, the number of induced copies of $D$ in $G[V_i]$ is at most $\delta \beta^3 |V_i|^3/32$, as required by Lemma \ref{lem:extract_clique}. Therefore, there is a set $X_{i+1} \subseteq V_i$ satisfying $d(X_{i+1}) \geq 1-\delta$ and 
	$|X_{i+1}| \geq \beta |V_i|/2 = \frac{\gamma n^2}{2|V_i|} \geq \gamma n/2$. Set $V_{i+1} = V_i \setminus X_{i+1}$. This process eventually terminates, resulting in the desired sequence $X_1,\dots,X_m,Y$.
\end{proof}

Let $G$ be an ordered graph. For disjoint $A,B \subseteq V(G)$, we write $A < B$ to mean that $a < b$ for all $a \in A, b\in B$. A subset $A' \subseteq A$ is called a {\em suffix} of $A$ if for all $a_1,a_2 \in A$, if $a_1 \in A'$ and $a_1 < a_2$ then $a_2 \in A'$. Similarly, $A'$ is a {\em prefix} of $A$ if for all $a_1,a_2 \in A$, if $a_2 \in A'$ and $a_1 < a_2$ then $a_1 \in A'$. 

We now describe the structure of an induced $D$-free graph consisting of two cliques. Lemma \ref{lem:charac_simple} handles the case when one of the cliques precedes the other one. Lemma \ref{lem:charac_general} then takes care of the general case.
\begin{lemma}\label{lem:charac_simple}
	Let $G$ be an ordered graph and let $A,B \subseteq V(G)$ be disjoint cliques with $A < B$. Then $G[A \cup B]$ is induced $D$-free if and only if for every $b \in B$, $N_A(b)$ is a suffix of $A$. 
\end{lemma}
\begin{proof}
	Using that $A$ and $B$ are cliques, if $G$ contains three vertices $x<y<z$ such that $\{x,y\},\{x,z\}\in E(G)$ and $\{y,z\}\not\in E(G)$, then we must have $x,y\in A$ and $z\in B$. For a given $z$, the existence of such $x$ and $y$ is equivalent to the statement that $N_{A}(z)$ is not a suffix of $A$. 
\end{proof}
\begin{lemma}\label{lem:charac_general}
	Let $G$ be an ordered graph and let $A,B \subseteq V(G)$ be disjoint cliques. Then $G[A \cup B]$ is induced $D$-free if and only if there is a partition $V(G) = I \cup J$ into intervals with $I < J$, such that the following holds:
%	\begin{enumerate}
%		\item There is an empty bipartite graph between $A \cap I, B \cap I$, and a complete bipartite graph between $A \cap J, B \cap J$. 
%		\item For every $b \in B \cap J$, $N_{A \cap I}(b)$ is a suffix of $A \cap I$, and for every $a \in A \cap J$, $N_{B \cap I}(a)$ is a suffix of $B \cap I$.
%	\end{enumerate}
	\begin{enumerate}
		\item The bipartite graph between $A \cap I$ and  $B \cap I$ is empty. 
		\item The bipartite graph between $A \cap J$ and $B \cap J$ is complete, i.e. $(A \cup B) \cap J$ is a clique.
%		\item $(A \cup B) \cap J$ is a clique. 
		\item For every $b \in B \cap J$, $N_{A \cap I}(b)$ is a suffix of $A \cap I$. \item For every $a \in A \cap J$, $N_{B \cap I}(a)$ is a suffix of $B \cap I$.
	\end{enumerate}
\end{lemma}
\begin{proof}
	We first check that if there is a partition $V(G) = I \cup J$ satisfying 1-4 then $G[A \cup B]$ is induced $D$-free. Suppose that $x < y < z$ form an induced copy of $D$. Without loss of generality, suppose that $x \in A$. If $x \in J$ then $y,z \in J$, but $(A \cup B) \cap J$ is a clique, contradicting that $\{y,z\} \notin E(G)$. So we must have $x \in I$. If $z\in I$, then $y\in I$, but $I\cap (A\cup B)$ is the disjoint union of cliques by item 1, so $G[I\cap (A\cup B)]$ cannot contain an induced copy of $D$. Hence, $z\in J$. But then $y\in I$ by item 2. This also implies $y\in A$, as otherwise $\{x,y\}$ is a non-edge by item 1. It now follows that $z\in B$, as $\{y,z\}$ is a non-edge and $A$ is a clique. However, this setup runs into a contradiction with item 3, so $x,y,z$ cannot induce a copy \nolinebreak of \nolinebreak $D$.

	Now we show that if $G[A \cup B]$ is induced $D$-free then there is a partition $V(G) = I \cup J$ satisfying 1-4. Take $J$ to be a maximal suffix of $V(G)$ with the property that the bipartite graph between $A \cap J$ and $B \cap J$ is complete. So item 2 holds by definition. If $J = V(G)$ then we are done. Otherwise, setting $I = V(G) \setminus J$, we show that the bipartite graph between $A \cap I$ and $B \cap I$ is empty. 
	Let $x$ be the largest element of $(A \cup B) \cap I$, and suppose without loss of generality that $x \in A$. By the maximality of $J$, there is $y \in B \cap J$ with $\{x,y\} \notin E(G)$.
	Suppose, for contradiction, that $\{a,b\} \in E(G)$ for some $a \in A \cap I$, $b \in B \cap I$. Assume first that $a < b$. If $\{x,b\} \notin E(G)$ then $a,b,x$ form an induced copy of $D$, and if $\{x,b\} \in E(G)$ then $b,x,y$ form an induced copy of $D$. Assume now that $b < a$. If $\{a,y\} \notin E(G)$ then $b,a,y$ form an induced copy of $D$, and if $\{a,y\} \in E(G)$ then $a,x,y$ form an induced copy of $D$. This proves item 1. Items 3-4 follow from Lemma \ref{lem:charac_simple}. 
\end{proof} 
\noindent
We now prove the removal versions of Lemmas \ref{lem:charac_simple} and \ref{lem:charac_general}. 
\begin{lemma}\label{lem:removal clique_pair simple}
	Let $G$ be an ordered graph and let $A,B \subseteq V(G)$ be disjoint cliques with $A < B$. For $\gamma > 0$, if $G[A \cup B]$ contains at most $\gamma^2|A|^2|B|/4$ induced copies of $D$, then $G[A \cup B]$ can be made induced $D$-free by adding/deleting at most $\gamma |A||B|$ edges between $A$ and $B$.
\end{lemma}
\begin{proof}
	Fix any $b \in B$. If $b$ has at most $\gamma |A|/2$ non-neighbours in $A$, then add all edges between $b$ and $A$. 
	% Doing this for all $b \in B$, we make a total of at most $\gamma|A||B|/3$ edge changes.
	Suppose now that $b$ has at least $\gamma|A|/2$ non-neighbours in $A$, and let $A_b$ be a suffix of $A$ such that $b$ has exactly $\gamma |A|/2$ non-neighbours in $A_b$. Observe that if $a \in A \setminus A_b$ is a neighbour of $b$ and $a' \in A_b$ is a non-neighbour of $b$, then $a,a',b$ span an induced $D$. So, letting $d_b := |N_{A \setminus A_b}(b)|$, we see that $b$ participates in at least 
	$d_b \cdot \gamma |A|/2$ induced copies of $D$ of this form. Summing over all $b \in B$, we get $\gamma|A|/2 \cdot \sum_{b \in B}{d_b}$ induced copies of $D$. Hence, by assumption, $\gamma|A|/2 \cdot \sum_{b \in B}{d_b} \leq \gamma^2|A|^2|B|/4$ and so $\sum_{b \in B}{d_b} \leq \gamma |A||B|/2$. For each $b \in B$, add all edges between $b$ and $A_b$ and delete all edges between $b$ and $A \setminus A_b$. This way we make at most $\gamma|A|/2 + d_b$ edge changes for each $b \in B$, resulting in a total of at most $\gamma|A||B|/2 + \sum_{b \in B}{d_b} \leq \gamma|A||B|$ edge changes.
	% in this second step. Together with the first step, the total number of edge changes is at most $\gamma|A||B|$. 
	By construction, after these changes  $N_A(b)$ is a suffix of $A$ for every $b \in B$. Hence, $G[A \cup B]$ is induced $D$-free by Lemma \ref{lem:charac_simple}.
\end{proof}
\begin{lemma}\label{lem:removal clique_pair general}
	Let $G$ be an ordered graph and let $A,B \subseteq V(G)$ be disjoint cliques. Let $\gamma > 0$ and suppose that $|A|,|B| \geq 8/\gamma$. If $G[A \cup B]$ contains at most $\frac{\gamma^2}{64} |A||B| \min\{|A|,|B|\}$ induced copies of $D$, then $G[A \cup B]$ can be made induced $D$-free by adding/deleting at most $\gamma |A||B|$ edges between $A$ and $B$.
\end{lemma}
\begin{proof}
	If $\bar{e}(A,B) \leq \gamma |A||B|$ then add all edges between $A$ and $B$ to make $G[A \cup B]$ a clique and hence induced $D$-free. Suppose now that $\bar{e}(A,B) > \gamma|A||B|$. Let $J$ be the minimal suffix of $V(G)$ with the property that $\bar{e}(A \cap J, B \cap J) \geq \gamma |A||B|/8$. By minimality, $$\bar{e}(A \cap J, B \cap J) \leq \gamma|A||B|/8 + \max\{|A|,|B|\} \leq \gamma |A||B|/4.$$ Let $I = V(G) \setminus J$. If $\{a_1,b_1\} \in E(A \cap I, B \cap I)$ and $\{a_2,b_2\} \in \bar{E}(A \cap J, B \cap J)$, then $G[\{a_1,a_2,b_1,b_2\}]$ contains an induced copy of $D$; this follows from Lemma \ref{lem:charac_general} applied to $\{a_1,a_2\},\{b_1,b_2\}$. Each induced copy of $D$ is counted at most $\max\{|A|,|B|\}$ times this way. Hence, $G$ contains at least 
	$$
	\frac{e(A \cap I, B \cap I) \cdot \bar{e}(A \cap J, B \cap J)}{\max\{|A|,|B|\}} \geq \frac{e(A \cap I, B \cap I) \cdot \gamma|A||B|}{8\max\{|A|,|B|\}}
	$$ 
	induced copies of $G$. Therefore, by our assumption on the number of induced copies of $D$, we \nolinebreak get $$e(A \cap I, B \cap I) \leq \gamma|A||B|/8.$$ Make the bipartite graph between $A \cap J$ and $B \cap J$ complete, and the bipartite graph between $A \cap I$ and $B \cap I$ empty. This requires at most $e(A \cap I, B \cap I) + \bar{e}(A \cap J, B \cap J) \leq \gamma|A||B|/2$ edge changes altogether. Let $s$ be the number of edge changes required to make $G[(A \cap I) \cup (B \cap J)]$ induced $D$-free. By Lemma \ref{lem:removal clique_pair simple}, applied to the two cliques $A \cap I< B \cap J$ and with parameter $$\beta := \frac{s}{|A \cap I||B \cap J|}$$ 
	instead of $\gamma$, there are at least 
	$$\frac{\beta^2}{4} \cdot |A \cap I|^2 \cdot |B \cap J| = \frac{s^2}{4|B \cap J|}$$ induced copies of $D$ in $G$. By assumption, $\frac{s^2}{4|B \cap J|} \leq \gamma^2 |A|^2|B|/64$, and hence $s \leq \gamma|A||B|/4$. Therefore, $G[(A \cap I) \cup (B \cap J)]$ can be made induced $D$-free by adding/deleting at most $\gamma|A||B|/4$ edges. Symmetrically, the same is true for $G[(B \cap I) \cup (A \cap J)]$. With these changes and the ones done in the previous step, the total number of edge additions/deletions is at most $\gamma|A||B|$. After the changes, items 1-4 of Lemma \ref{lem:charac_general} are satisfied, and hence $G[A \cup B]$ is induced $D$-free. 
\end{proof}
For an ordered set $V$ and disjoint subsets $A_1,\dots,A_k \subseteq V$, an {\em $(A_1,\dots,A_k)$-sequence} is a sequence $v_1 < v_2 < \dots < v_k$ of elements of $V$ with $v_i \in A_i$. The next lemma we need is a removal lemma for ordered sequences, which follows from the main result of \cite{RR}.

\begin{lemma}\label{lem:sequence_removal}
	Let $k \geq 1$, then there exists $c_{k}>0$ such that the following holds. Let $V$ be an ordered set of size $n$ and let $A_1,\dots,A_k \subseteq V$ be disjoint subsets. For $\varepsilon > 0$, if the number of $(A_1,\dots,A_k)$-sequences is at most $c_k \varepsilon^k n^k$, then there is a set $S \subseteq V$, $|S| \leq \varepsilon n$, which intersects every $(A_1,\dots,A_k)$-sequence.  
\end{lemma}

What was actually proved in \cite{RR} is that if a sample of $q = O(k/\varepsilon)$ elements of $V$, taken uniformly at random and independently, contains no $(A_1,\dots,A_k)$ sequence with probability larger than $\frac{1}{3}$, then there is a set $S \subseteq V$, $|S| \leq \varepsilon n$, which intersects every $(A_1,\dots,A_k)$-sequence. Observe that if the number of $(A_1,\dots,A_k)$-sequences is $N$, then the probability that such a sample contains an $(A_1,\dots,A_k)$-sequence is at most $(q/n)^k N$. Hence, if $N \leq c_k\varepsilon^kn^k$ then this probability is less than $2/3$ (with an appropriate choice of $c_k$), and Lemma \ref{lem:sequence_removal} follows.  

We now move on to the following lemma, which is one of the main ingredients in the proof of Theorem \ref{thm:P3}. 

\begin{lemma}\label{lem:interval_partition}
	Let $c=c_3/64$, where $c_3$ is the constant defined in Lemma \ref{lem:sequence_removal}. Let $G$ be an $n$-vertex ordered graph with a vertex-partition $V(G) = X_1 \cup \dots \cup X_m$ such that $X_1,\dots,X_m$ are cliques and $G[X_i \cup X_j]$ is induced $D$-free for all $i < j$. 
	Let $\gamma > 0$ and suppose that $G$ has at most 
	% $c\gamma^6 n^3/64m^{15}$ 
	$c \gamma^6n^3/m^{15}$
	induced copies of $D$.
	Then there is a set $S \subseteq V(G)$, $|S| \leq \gamma n$, and a partition of $V(G)\setminus S$ into intervals $I_1,\dots,I_t$, $t \leq 2m^3$, such that the following holds: for every $1 \leq j \leq t$, $G[I_j]$ is a disjoint union of cliques, each of the form $I_j \cap(\bigcup_{i \in M}{X_i})$ for some $M \subseteq [m]$. 
\end{lemma}
\begin{proof}
	For each $1 \leq i < j \leq m$, Lemma \ref{lem:charac_general} states that there is a partition $V(G) = I_{i,j} \cup J_{i,j}$ into intervals such that the bipartite graph between $X_i \cap I_{i,j}$ and $X_j \cap I_{i,j}$ is empty, and the bipartite graph between $X_i \cap J_{i,j}$ and $X_j \cap J_{i,j}$ is complete. Let $I'_1,\dots,I'_s$ be the common refinement of the partitions $\left\{ I_{i,j},J_{i,j} : 1 \leq i < j \leq m \right\}$. Then $s\leq \binom{m}{2}+1 \leq m^{2}/2$, as the right endpoints of the intervals $I_{i,j}$ together with the last vertex of $V(G)$ are the right endpoints of the intervals $I'_1,\dots,I'_{s}$.
	Observe that for each $1 \leq \ell \leq s$ and $1 \leq i < j \leq m$, the bipartite graph between $X_i \cap I'_{\ell}$ and $X_j \cap I'_{\ell}$ is either complete or empty, because $I'_{\ell} \subseteq I_{i,j}$ or $I'_{\ell} \subseteq J_{i,j}$.
	
	Fix any $1 \leq \ell \leq s$. If $|I'_{\ell}| < \gamma n/m^2$ then put all vertices of $I'_{\ell}$ into $S$. 
	This puts at most $s \cdot \gamma n/m^2 \leq \gamma n/2$ vertices in $S$ altogether. 
	Suppose now that $|I'_{\ell}| \geq \gamma n/m^2$. 
	We run the following process. If there are distinct $1 \leq i_1,i_2,i_3 \leq m$ such that the number of $(X_{i_1},X_{i_2},X_{i_3})$-sequences inside $I'_{\ell}$ is at least $1$ and at most 
    $c_3 (\frac{\gamma}{2m^3})^3 |I'_{\ell}|^3$, then use Lemma \ref{lem:sequence_removal} to delete $\frac{\gamma}{2m^3} \cdot|I'_{\ell}|$ vertices from $I'_{\ell}$ and thus destroy all $(X_{i_1},X_{i_2},X_{i_3})$-sequences in $I'_{\ell}$. Add the deleted vertices to $S$ and update $I'_{\ell}$. Let $I''_{\ell} \subseteq I'_{\ell}$ be the interval at the end of the process. 
    Note that each triple $i_1,i_2,i_3$ can only cause vertex-deletion once (because following this vertex-deletion, there are no more $(X_{i_1},X_{i_2},X_{i_3})$-sequences in $I'_{\ell}$). 
    Hence, the total number of deleted vertices is at most $m^3 \cdot \frac{\gamma}{2m^3} \cdot|I'_{\ell}| = \frac{\gamma}{2}|I'_{\ell}|$. So $|I''_{\ell}| \geq (1 - \frac{\gamma}{2})|I'_{\ell}| \geq |I'_{\ell}|/2 \geq \gamma n/(2m^2)$. 
    Doing this step for every $1 \leq i \leq \ell$ adds a total of at most $\gamma/2 \cdot \sum_{\ell = 1}^s {|I'_{\ell}|} \leq \gamma n/2$ vertices to $S$. Thus, the total number of vertices in $S$ at this point is at most  $\gamma n$. 
    
    Observe that after this step, for every triple of distinct $1 \leq i_1,i_2,i_3 \leq m$, either there are no $(X_{i_1},X_{i_2},X_{i_3})$-sequences inside $I''_{\ell}$, or the number of these sequences is at least 
	$$
 	c_3 \left(\frac{\gamma}{2m^3}\right)^3 |I''_{\ell}|^3 \geq 
 	c_3 \left(\frac{\gamma}{2m^3}\right)^3 \cdot \left(\frac{\gamma n}{2m^2}\right)^3 = 
 	c \gamma^6 n^3/m^{15}.
 	$$ 
	
	For each $1 \leq i \leq m$, let $J^{\ell}_i$ be the minimal subinterval of $I''_{\ell}$ which contains the set $X_i \cap I''_{\ell}$. 
	By minimality, the first and last elements of $J^{\ell}_i$ belong to $X_i$. 
	Now take the common refinement of the intervals $J^{\ell}_1,\dots,J^{\ell}_m$, giving a partition of $I''_{\ell}$ into at most $2m+1$ intervals. Doing this for every $1 \leq \ell \leq s$, we get a partition $I_1,\dots,I_t$ of $V(G) \setminus S$ into $t \leq s \cdot (2m+1) \leq m^2/2 \cdot (2m+1) \leq 2m^3$ intervals. 
	For $1 \leq \ell \leq s$ and $1 \leq j \leq t$ with $I_j \subseteq I''_{\ell}$, observe that if $X_i \cap I_j \neq \emptyset$, then $I_j \subseteq J^{\ell}_i$, which means that there is an element of $X_i \cap I''_{\ell}$ which is smaller or equal to the first element of $I_j$ (indeed, the first element of $J^{\ell}_i$ satisfies this), as well as an element of $X_i \cap I''_{\ell}$ which is bigger or equal to the last element of $I_j$ (indeed, the last element of $J^{\ell}_i$ satisfies this). 
	
	Let us show that $I_1,\dots,I_t$ have the property stated in the lemma. Fix any $1 \leq j \leq t$. By construction, there is $1 \leq \ell \leq s$ such that $I_j$ is a subinterval of $I''_{\ell}$. Recall that for all $1 \leq i_1 \neq i_2 \leq m$, the bipartite graph between $X_{i_1} \cap I''_{\ell}$ and $X_{i_2} \cap I''_{\ell}$ is either complete or empty.
	
	If $G[I_j]$ is not a disjoint union of cliques then it contains an (unordered) path on 3 vertices. Since $X_1,\dots,X_m$ are cliques and all bipartite graphs $X_{i_1} \cap I''_{\ell}$ and $X_{i_2} \cap I''_{\ell}$ are complete or empty, this path cannot contain two vertices from the same clique. Hence, there are
    distinct $1 \leq i_1,i_2,i_3 \leq m$ such that $X_{i_1},X_{i_2},X_{i_3}$ all intersect $I_j$, the bipartite graphs 
% 	$(X_{i_1} \cap I_{j}, X_{i_2} \cap I_j)$ and $(X_{i_1} \cap I_{j}, X_{i_3} \cap I_{j})$ are complete, and the bipartite graph $(X_{i_2} \cap I_{j}, X_{i_3} \cap I_{j})$ is empty. But then we also have that the bipartite graphs 
	$(X_{i_1} \cap I''_{\ell}, X_{i_2} \cap I''_{\ell})$ and $(X_{i_1} \cap I''_{\ell}, X_{i_3} \cap I''_{\ell})$ are complete, and the bipartite graph $(X_{i_2} \cap I''_{\ell}, X_{i_3} \cap I''_{\ell})$ is empty. Since $X_{i_1}$ intersects $I_j$, there exists $x_{i_1} \in X_{i_1} \cap I''_{\ell}$ such that $x_{i_1}$ is smaller or equal to the first element of $I_j$. Similarly, since $X_{i_2}$ and $X_{i_3}$ intersect $I_j$, there exist $x_{i_2} \in X_{i_2} \cap I''_{\ell}$ and $x_{i_3} \in X_{i_3} \cap I''_{\ell}$ such that $x_{i_2},x_{i_3}$ are bigger or equal to the last element of $I_j$. Without loss of generality, suppose that $x_{i_2} < x_{i_3}$. Then $(x_{i_1},x_{i_2},x_{i_3})$ is a $(X_{i_1},X_{i_2},X_{i_3})$-sequence contained in $I''_{\ell}$. By construction, there are at least $c \gamma^6 n^3/m^{15}$ such sequences. Now observe that each such sequence spans an induced copy of $D$, contradicting the assumption of the lemma. This completes the proof.  
\end{proof}
In the next lemma we prove a $D$-removal lemma for graphs which can be partitioned into two intervals, each of which induces a disjoint union of cliques. An important feature is that edge changes are only made between the intervals, not inside them, so that each interval remains a disjoint union of cliques after the changes.  
\begin{lemma}\label{lem:fixing_pairs}
	Let $G$ be an ordered graph, let $I,J \subseteq V(G)$ be disjoint intervals with $I < J$, and suppose that each of the graphs $G[I],G[J]$ is the disjoint union of at most $m$ cliques. Let $\gamma > 0$, and suppose that $G[I \cup J]$ contains at most $\frac{\gamma^{15}}{2^{40}m^9}|I||J| \min\{|I|,|J|\}$ induced copies of $D$. Then $G[I \cup J]$ can be made induced $D$-free by adding/deleting at most $\gamma|I||J|$ edges between $I$ and \nolinebreak $J$. 
\end{lemma}
\begin{proof}
	Suppose that $G[I]$ is the disjoint union of cliques $A_1,\dots,A_k$ and $G[J]$ is the disjoint union of cliques $B_1,\dots,B_{\ell}$. By assumption, $k,\ell \leq m$. We will make $G[I \cup J]$ induced $D$-free in two steps.
	\paragraph{Step 1.} 
	Set 
	$$
	\delta := \frac{\gamma^6}{2^{16}m^3}.
	$$
	Fix $i \in [k]$ and $j \in [\ell]$ with $|A_i| \geq \frac{\gamma |I|}{4m}$ and $|B_j| \geq \frac{\gamma |J|}{4m}$. 
	Then $\delta^2|A_i|^2|B_j|/4 \geq \frac{\gamma^{15}}{2^{40}m^9}|I|^2|J|$.
	By assumption, the number of induced copies of $D$ in $G[A_i \cup B_j]$ is at most $\delta^2|A_i|^2|B_j|/4$. Apply Lemma \ref{lem:removal clique_pair simple} to make $G[A_i \cup B_j]$ induced $D$-free with at most $\delta|A_i||B_j|$ edge changes. By Lemma \ref{lem:charac_simple}, the neighbourhood in $A_i$ of each vertex in $B_j$ is now a suffix of $A_i$. Performing step 1 for all pairs $i \in [k], j \in [\ell]$ requires at most $\delta|I||J|$ edge changes altogether. 
	
	\paragraph{Step 2.} For every $i \in [k]$ and $j \in [\ell]$, delete all edges between $A_i$ and $B_j$ if either $e(A_i,B_j) \leq \gamma|A_i||B_j|/4$ or 
	$|A_i| \leq \frac{\gamma |I|}{4m}$ or $|B_j| \leq \frac{\gamma |J|}{4m}$. This step requires at most $\gamma |I||J|/4 + 2m \cdot \frac{\gamma |I||J|}{4m} = 3\gamma|I||J|/4$ edge changes overall. So together with step 1, the total number of edge changes is at most $\gamma|I||J|$. 
	
	\vspace{0.2cm}
	We claim that after step 2, $G[I \cup J]$ is induced $D$-free. Suppose that this is not the case. Note that $G[I]$ and $G[J]$ are induced $D$-free because each is a disjoint union of cliques. If an induced copy of $D$ has two vertices in $I$ and one in $J$, then it must be contained in $A_i \cup B_j$ for some $i,j$, since $A_1,\dots,A_k$ are cliques with no edges between them. However, in steps 1 and 2 we made sure that $G[A_i \cup B_j]$ is induced $D$-free for all $i,j$, so this is impossible. Hence, the induced $D$-copy must be of the form $a,b,b'$ with $a \in A_i$, $b \in B_j$ and $b' \in B_{j'}$ for some $i \in [k]$, $j,j' \in [\ell]$. Since $a$ is adjacent to $b$ and $b'$, we did not make any edge changes in the bipartite graphs $(A_i,B_j)$ and $(A_i,B_{j'})$ in step 2. Hence,
	it must be the case that
	\begin{equation}\label{eq:A_i, B_j, B_j' vertices}
	|A_i| \geq \frac{\gamma |I|}{4m}, \; \; |B_j|,|B_{j'}| \geq \frac{\gamma |J|}{4m}
	\end{equation}
	and
	\begin{equation}\label{eq:A_i, B_j, B_j' edges}
	e(A_i,B_j) \geq \frac{\gamma|A_i||B_j|}{4}, \; \; e(A_i,B_{j'}) \geq \frac{\gamma|A_i||B_{j'}|}{4}.
	\end{equation}
	Since $b$ and $b'$ are not adjacent, $j \neq j'$. There are no edges between $B_j$ and $B_{j'}$. 
	Recall that the neighbourhood in $A_i$ of each vertex in $B_j \cup B_{j'}$ is a suffix of $A_i$. This means that if $a_1,a_2 \in A_i$ and $a_1 < a_2$, then $N_{B_j}(a_1) \subseteq N_{B_j}(a_2)$ and $N_{B_{j'}}(a_1) \subseteq N_{B_{j'}}(a_2)$. Let $A'$ be the set of the last $\gamma|A_i|/8$ elements of $A_i$. 
	Due to \eqref{eq:A_i, B_j, B_j' edges}, 
	there exist at least $\gamma |A_i|/8$ vertices $a' \in A$ with $|N_{B_j}(a')| \geq \gamma|B_j|/8$. Since $N_{B_j}(a_1) \subseteq N_{B_j}(a_2)$ for $a_1 < a_2$, we must have that $|N_{B_{j}}(a')| \geq \gamma|B_j|/8$ for every $a' \in A'$. Similarly, $|N_{B_{j'}}(a')| \geq \gamma|B_{j'}|/8$ for every $a' \in A'$. 
%
% 	we must have $|N_{B_{j}}(a')| \geq \gamma|B_j|/8$ and $|N_{B_{j'}}(a')| \geq \gamma|B_{j'}|/8$ for every $a' \in A'$. 
	%Every triple $a' \in A_i, b_1 \in N_{B_j}(a), b_2 \in N_{B_{j'}}(a)$ forms an induced copy of $D$. 
	Every $a' \in A$ forms an induced copy of $D$ with every pair in $N_{B_{j}}(a') \times N_{B_{j'}}(a')$.
	This gives a total of at least
	$$
	|A'| \cdot \frac{\gamma|B_j|}{8} \cdot \frac{\gamma|B_{j'}|}{8} 
	= \frac{\gamma^3}{2^9} \cdot |A_i| \cdot |B_j| \cdot |B_{j'}|
	\geq \frac{\gamma^3}{2^9} \cdot \frac{\gamma |I|}{4m} \cdot \left( \frac{\gamma |J|}{4m} \right)^2 = \frac{\gamma^6}{2^{15}m^3}|I||J|^2 = 2\delta|I||J|^2
	$$
	induced copies of $D$, each having one vertex in $I$ and two in $J$. The at most $\delta|I||J|$ edges we added/deleted in step 1 can participate in at most $\delta|I||J|^2$ of these copies. Hence, at least $\delta|I||J|^2$ must be present in the original graph, a contradiction to the assumption of the lemma.  
\end{proof}

In the following lemma we show that under certain conditions, if there are few induced $D$-copies of a certain type, then they can be destroyed by deleting few vertices (rather than by adding/deleting \nolinebreak edges). This is useful because vertex deletion, as opposed to edge changes, cannot create new induced copies.  

\begin{lemma}\label{lem:three_cliques vertex_removal}
	Let $G$ be an ordered graph and let $A,B,C \subseteq V(G)$ be disjoint cliques with $A < B < C$, whose union is $V(G)$. Suppose that $G[A \cup B]$, $G[A \cup C]$ and $G[B \cup C]$ are induced $D$-free. For $s \geq 1$, if $G[A \cup B \cup C]$ has at most $s^3/12$ induced copies of $D$, then there is a set $S$ of at most $3s$ vertices whose deletion destroys every induced copy of $D$ in $G[A \cup B \cup C]$.
\end{lemma}
\begin{proof}
	Clearly, if $G$ contains a copy of $D$ with vertices $a<b<c$, then $a\in A, b\in B, c\in C$. Let $(a_1,b_1,c_1),\dots,(a_r,b_r,c_r)$ be a maximum collection of vertex-disjoint induced copies of $D$ with $a_i \in A$, $b_i \in B$, $c_i \in C$ ($i = 1,\dots,r$). It is enough to show that $r \leq s$, because then $S=\{a_{i},b_{i},c_{i}:i\in [r]\}$ suffices. Assume without loss of generality that $a_1 < \dots < a_r$. By Lemma \ref{lem:charac_simple}, the neighbourhood  of each vertex of $B \cup C$ in $A$ is a suffix of $A$ (as $G[A \cup B]$ and $G[A \cup C]$ are induced $D$-free). For each $1 \leq i \leq r$, $a_i$ is adjacent to $b_i$ and $c_i$. 	Hence, for each pair of indices $1 \leq j \leq i \leq r$, $a_i$ is adjacent to $b_j$ and $c_j$. 
	
	For each $1 \leq i \leq r$, $\{b_i,c_i\} \notin E(G)$. This implies that for every $1 \leq j < k \leq r$, at least one of the pairs $\{b_j,c_k\},\{c_j,b_k\}$ must be a non-edge. Indeed, otherwise $b_j,c_j,b_k,c_k$ span an induced 4-cycle, which must contain an induced copy of $D$, contradicting the assumption that $G[B \cup C]$ is induced $D$-free. We have shown that for every triple of indices $1 \leq j < k \leq i \leq r$, $a_i$ forms an induced copy of $D$ with one of the pairs $\{b_j,c_k\}$ or $\{c_j,b_k\}$. This gives a total of at least $\sum_{i=1}^r {\binom{i}{2}} = \binom{r+1}{3} \geq r^3/12$ induced copies of $D$. By the assumption of the lemma, $r \leq s$, as required.
\end{proof}

In the following lemma we prove a $D$-removal statement in the following setting. Suppose that $V(G) = X \cup Y$, $Y$ is independent, and $G[X]$ is induced $D$-free. Then one can efficiently destroy the remaining induced $D$-copies by deleting edges between $X$ and $Y$. 

\begin{lemma}\label{lem:P2 independet_set}
	Let $G$ be an ordered graph and let $V(G) = X \cup Y$ be a vertex partition such that $Y$ is an independent set and $G[X]$ is induced $D$-free. Let $\gamma > 0$ with 
% 	$|X|,|Y| \geq 8/\gamma$. 
    $|Y| \geq 4/\gamma$. 
	Suppose that $G$ has at most $\frac{\gamma^2}{32}|X||Y|\min\{|X|,|Y|\}$ induced copies of $D$. Then $G$ can be made induced $D$-free by deleting at most $\gamma |X||Y|$ edges between $X$ and $Y$.
\end{lemma}
\begin{proof}
	We delete edges in two steps. For each $x \in X$, let $d_x$ be the number of $y \in Y$, $y > x$, such that $\{x,y\} \in E(G)$. Let $e := \sum_{x \in X}{d_x}$ be the number of edges in which the first vertex is in $X$ and the second in $Y$. We claim that $e \leq \gamma|X||Y|/2$. If $e \leq 2|X|$ then this holds because $|Y| \geq 4/\gamma$ by assumption. Suppose that $e \geq 2|X|$. Observe that as $Y$ is an independent set, each $x$ participates in $\binom{d_x}{2}$ induced copies of $D$, in which the other two vertices are from $Y$. By using our assumption on the one hand and Jensen's inequality on the other, we get 
	$$\frac{\gamma^2}{32}|X||Y|^2 \geq
	\sum_{x \in X}\binom{d_x}{2} \geq
	|X|\binom{e/|X|}{2} \geq \frac{e^2}{4|X|},
	$$
	where the last inequality uses $e \geq 2|X|$.
So indeed $e \leq \gamma|X||Y|/2$. Delete all edges in which the first vertex is in $X$ and the second in $Y$. This destroys all induced $D$-copies in which the first vertex is in $X$ (since $G[X]$ is induced $D$-free by assumption). 
	
	Next, for each $y \in Y$, let $N_y \subseteq X$ be the set of all $x \in X$, $x > y$, with $\{x,y\} \in E(G)$. Let $M_y$ be a largest matching of non-edges inside $N_y$. 
	% We claim that $\sum_{y \in Y}|M_y| \leq \gamma|X||Y|/4$. If $\sum_{y \in Y}|M_y| \leq 2|Y|$ then this holds because $|X| \geq 8/\gamma$ by assumption. Suppose now that $\sum_{y \in Y}|M_y| \geq 2|Y|$. 
	Observe that for every $\{x_1,x_2\},\{x_3,x_4\} \in M_y$, at least one of the pairs $\{x_1,x_3\},\{x_1,x_4\},\{x_2,x_3\},\{x_2,x_4\}$ must be a non-edge. Indeed, otherwise $x_1,x_2,x_3,x_4$ span an induced 4-cycle, which must contain an induced copy of $D$, in contradiction to the assumption that $G[X]$ is induced $D$-free. Note that $y$ forms an induced $D$-copy with every non-edge inside $N_y$. Hence, $y$ is part of at least $|M_y| + \binom{|M_y|}{2} \geq \frac{|M_y|^2}{2}$ induced $D$-copies, in which the other two vertices are from $X$. Using our assumption and Jensen's inequality, we get
% 	$$
% 	\frac{\gamma^2}{64}|X|^2|Y| \geq
% 	\sum_{y \in Y}\binom{|M_y|}{2} \geq
% 	|Y| \cdot \binom{\sum_{y \in Y}|M_y|/|Y|}{2} \geq \frac{\left( \sum_{y \in Y}|M_y| \right)^2}{4|Y|},
% 	$$
    $$
	\frac{\gamma^2}{32}|X|^2|Y| \geq
	\sum_{y \in Y}\frac{|M_y|^2}{2} \geq
	\frac{|Y|}{2} \cdot \left(\frac{1}{|Y|}\sum_{y \in Y}|M_y|\right)^2 = \frac{\left( \sum_{y \in Y}|M_y| \right)^2}{2|Y|},
	$$
Hence $\sum_{y \in Y}|M_y| \leq \gamma|X||Y|/4$. For each $y \in Y$, delete the edges between $y$ and $x_1,x_2$ for every $\{x_1,x_2\} \in M_y$. By the choice of $M_y$, after this step the set $\{x \in X : x > y, \{x,y\} \in E(G)\}$ is a clique. Hence, this step destroys all induced $D$-copies in which the first vertex is in $Y$ (recall that $Y$ is an independent set). The number of edge deletions in this second step is $2\sum_{y \in Y}|M_y| \leq \gamma|X||Y|/2$. Hence, the total number of edges deleted is at most $\gamma|X||Y|$. 
\end{proof}

\noindent
We are finally in a position to prove Theorem \ref{thm:P3}. 

\begin{proof}[Proof of Theorem \ref{thm:P3}]
	We may and will assume that $n$ is larger than some suitable polynomial function of $1/\varepsilon$. Set 
	$$
	\varepsilon_1 := \frac{\varepsilon^5}{1000},
	$$
	$$
	\varepsilon_2 := \frac{\varepsilon^{3}\cdot\varepsilon_1^{36}}{2^{100}},
	$$
	$$
	\varepsilon_3 := 
	\min\left\{ \frac{c\cdot\varepsilon^{6}\cdot\varepsilon_1^{15}}{2^{40}}, \; \;
	\frac{\varepsilon_1^{18}\cdot \varepsilon_{2}^{18}}{2^{100}}
	\right\},
	$$
	and
	$$
	\delta := \frac{\varepsilon_3^2 \cdot \varepsilon_1^3}{512} \; .
	$$
	Here, $c$ is the constant from Lemma \ref{lem:interval_partition}. We show that if $G$ contains less than $N:=\delta\varepsilon_1^3n^{3}/32$ induced copies of $D$, then we can make $G$ induced $D$-free by adding/deleting at most $\varepsilon n^{2}$ edges.
	
	First, apply Lemma \ref{lem:partition} with parameters $\gamma = \varepsilon_1$ and $\delta$. As $G$ contains less than $\delta\varepsilon_1^3 n^3/32 $ induced copies of $D$, there exists a partition
	$V(G) = X_1 \cup \dots \cup X_m \cup Y$ satisfying the following properties.
	\begin{enumerate}
	    \item $|X_i|\geq \varepsilon_1 n/2$ and $d(X_i)\geq 1-\delta$ for $i\in [m]$,
	    \item $e(Y)\leq \varepsilon_{1} n^{2}$.
	\end{enumerate}
	  Note that $m \leq 2/\varepsilon_1$. Let $X=X_1\cup\dots\cup X_{m}$. If $X=\emptyset$, then we can make $G$ induced $D$-free by deleting all the edges, which is at most $\varepsilon_1 n^{2}\leq \varepsilon n^{2}$ changes. So we can assume that $X\neq \emptyset$, which implies that $|X|\geq \varepsilon_1 n/2$. We will modify $G$ in 6 steps.
	  
	\paragraph{Step 1.} Make $X_1,\dots,X_m$ cliques and $Y$ an independent set. This requires at most $\delta \cdot \nolinebreak \sum_{i=1}^m{\binom{|X_i|}{2}} \leq \delta \binom{|X|}{2} \leq \delta |X|^2/2$ edge changes inside $X := X_1 \cup \dots \cup X_m$, and at most $\varepsilon_1 n^2$ edge changes inside $Y$. Set $K_1 := \delta |X|^2/2 + \varepsilon_1 n^2$. 
	Denote the resulting graph by $G_1$. Note that every modified edge participates in at most $n$ induced copies of $D$, and in at most $|X|$ induced copies of $D$ which are contained in $X$.  
	% so $G_1$ contains at most $N_{1}=N+\delta n^{3}/2+\varepsilon_1 n^{3}<2\varepsilon_1 n^{3}$ induced copies of $D$. Also,
	Hence, $G_1[X]$ contains at most  $M_{1}=N+\delta |X|^{3}/2 < \delta |X|^{3}$ induced copies of $D$.  
	
	\paragraph{Step 2.} For every $1 \leq i < j \leq m$, make $G_1[X_i \cup X_j]$ induced $D$-free. By Lemma \ref{lem:removal clique_pair general}, this can be done by at most $\varepsilon_3 |X_i||X_j|$ edge changes. Indeed, otherwise $G_1[X_i \cup X_j]$ contains at least 
	$$
	\frac{\varepsilon_3^2}{64}|X_i||X_j|\min \{|X_i|,|X_j|\} \geq \frac{\varepsilon_3^2}{64} \cdot \left(\frac{\varepsilon_1 n}{2}\right)^3 = \delta n^3
	$$
	induced copies of $D$, contradiction. Executing these changes for all pairs $1 \leq i < j \leq m$ requires adding/deleting at most $\varepsilon_3 \binom{|X|}{2} \leq \varepsilon_3 n^2/2=:K_2$ edges altogether. Let $G_2$ be the resulting graph. Note that $G_{2}[X]$ contains at most $M_2=M_1+\varepsilon_3 |X|^3/2<\varepsilon_3 |X|^{3}$ induced copies of $D$.
	
	\vspace{0.3cm}
	\noindent
	Let us now apply Lemma \ref{lem:interval_partition} to $G_{2}[X]$ with parameter $\gamma = \varepsilon/12$. Note that 
	$$
	c\left(\frac{\varepsilon}{12}\right)^6 \frac{|X|^3}{m^{15}} \geq c\left(\frac{\varepsilon}{12}\right)^6 \frac{|X|^3}{(2/\varepsilon_1)^{15}} \geq \varepsilon_{3} |X|^{3} >M_2.
	$$
	Hence, $G_{2}[X]$ contains at most $c\left(\frac{\varepsilon}{12}\right)^6\frac{|X|^3}{m^{15}}$
	induced copies of $D$, meaning that the condition of Lemma \ref{lem:interval_partition} is satisfied. Therefore, there is a set $S \subseteq X$ of size at most $\varepsilon n/12$ and a partition of $X \setminus S$ into $t\leq 2m^{3}\leq 16/\varepsilon_1^{3}$ intervals $I_1,\dots,I_t$ such that for every $j\in [t]$, $G_2[I_j]$ is the disjoint union of cliques, each clique of the form 
	$I_{j}\cap (\bigcup_{i\in M} X_i)$ for some $M\subset[m]$. It follows that $G_2[I_j]$ is induced $D$-free. From this point on, we will make no edge changes inside the sets $I_1,\dots,I_t$, so this property will continue to hold.
	
	\paragraph{Step 3.}
	We will make $G_2[I_i \cup I_j]$ induced $D$-free for every $1 \leq i < j \leq t$, as follows. If $|I_i|$ or $|I_j|$ is smaller than $\frac{\varepsilon_2 |X|}{4t}$, then delete all edges between $I_i$ and $I_j$. Doing this for all such pairs $1 \leq i < j \leq t$ requires at most $t \cdot \frac{\varepsilon_2 |X|}{4t} \cdot |X| = \frac{\varepsilon_2|X|^2}{4}$ edge changes altogether. 
	Now fix a pair $1 \leq i < j \leq t$ with $|I_i|,|I_j| \geq \frac{\varepsilon_2 |X|}{4t}$. Apply Lemma \ref{lem:fixing_pairs} to $I_i,I_{j}$  with parameter $\gamma = \varepsilon_2$. We \nolinebreak have
	$$
	\frac{\varepsilon_2^{15}}{2^{40}m^{9}} |I_i||I_j| \min\{|I_i|,|I_j|\} \geq 
	\frac{\varepsilon_2^{15}}{2^{40}m^{9}} \left( \frac{\varepsilon_2 |X|}{4t} \right)^3 \geq 
	\frac{\varepsilon_2^{18}}{2^{40}(2/\varepsilon_1)^{9}(64/\varepsilon_1^3)^3} |X|^3 \geq 
	\varepsilon_3 |X|^3>M_{2}
	$$
	so the number of induced copies of $D$ in $G[I_{i}\cup I_{j}]$ satisfies the required condition of the lemma. Therefore, we can make $G_2[I_i \cup I_j]$ induced $D$-free by adding/deleting at most $\varepsilon_2|I_i||I_j|$ edges between $I_i$ and $I_j$. Altogether, in step 3 we make at most 
	$\varepsilon_2|X|^2/4 + \varepsilon_2 \sum_{1 \leq i < j \leq t}{|I_i||I_j|} \leq \varepsilon_2|X|^2/4 + \varepsilon_2 \binom{|X|}{2} \leq 3\varepsilon_2 |X|^2/4=:K_3$ edge changes. Denote the resulting graph by $G_3$. Note that the number of induced copies of $D$ in $G_{3}[X]$ is at most $M_3=M_{2}+3\varepsilon_{2}|X|^{3}/4<\varepsilon_2|X|^{3}$.
	% All in all, $G[X]$ and $G_3[X]$ differ in at most $(3\varepsilon_2/4 + \varepsilon_3)|X|^2 \leq \varepsilon_2 |X|^2$ edges. 
	%Recalling the at most $\varepsilon_1 n^2$ edges inside $Y$ that we deleted in step 1, we see that $G$ and $G_3$ differ in at most \nolinebreak $2\varepsilon_1 n^2$ \nolinebreak edges.  
	
	\paragraph{Step 4.} Fix any $1 \leq j_1 < j_2 < j_3 \leq t$ and $1 \leq i_1,i_2,i_3 \leq m$. Apply Lemma \ref{lem:three_cliques vertex_removal} to the cliques $X_{i_1} \cap I_{j_1},X_{i_2} \cap I_{j_2},X_{i_3} \cap I_{j_3}$ with parameter $s = \frac{\varepsilon n}{3m^{3}t^3}$. The number of induced copies of $D$ in $G_{3}[(X_{i_1} \cap I_{j_1})\cup(X_{i_2} \cap I_{j_2})\cup(X_{i_3} \cap I_{j_3})]$ is at most $M_{3}$, and we have
	$$
	\frac{s^3}{12}=\frac{1}{12} \cdot \left(\frac{\varepsilon n}{3m^3t^3} \right)^3 \geq 
	\frac{1}{12} \cdot \left(\frac{\varepsilon n}{3 \cdot (2/\varepsilon_1)^3(16/\varepsilon_1^{3})^3} \right)^3 > \varepsilon_2 n^3 > M_{3}.
	$$
	So the condition in Lemma \ref{lem:three_cliques vertex_removal} is satisfied.
    Therefore, there is a set $S_{j_1,j_2,j_3,i_1,i_2,i_3}$ 
	of size $3s=\frac{\varepsilon n}{m^3t^3}$ which intersects each such induced $D$-copy. Add the elements of $S_{j_1,j_2,j_3,i_1,i_2,i_3}$ to $S$. Doing this for every $1 \leq j_1 < j_2 < j_3 \leq t$ and $1 \leq i_1,i_2,i_3 \leq m$ increases the size of $S$ by at most $\binom{t}{3}m^3 \cdot \frac{\varepsilon n}{m^3t^3} \leq \varepsilon n/6$. Hence, after this step we have $|S| \leq \varepsilon n/12 + \varepsilon n/6 = \varepsilon n/4$. 
%	
%	\vspace{0.2cm}
%	\noindent
	Observe that after step 4, there are no induced copies of $D$ in $G_3[X \setminus S]$. 
	
%	\vspace{0.3cm}
%	\noindent
%	Note that $G[(X \setminus S) \cup Y]$ and $G_3[(X \setminus S) \cup Y]$ differ in at most $\varepsilon_2 |X|^2 + \varepsilon_1 n^2 \leq 2\varepsilon_1 n^2$ edges.
	
	\paragraph{Step 5.}
	Delete all edges touching the vertices in $S$. This requires at most $|S|n \leq \varepsilon n^2/4=:K_5$ edge changes. Note that if $|Y| \leq \varepsilon n/2$, then by deleting all edges touching $Y$, of which there are at most $|Y|n \leq \varepsilon n^2/2$, we make the graph induced $D$-free. Since the total number of edge changes in all previous steps is at most $K_1+K_2+K_3+K_5<2\varepsilon_1 n^2 + \varepsilon n^2/4 \leq \varepsilon n^2/2$, this would contradict the assumption that $G$ is $\varepsilon$-far from being induced $D$-free. Hence, $|Y| \geq \varepsilon n/2$. By similar reasoning, we may assume that $|X \setminus S| \geq \varepsilon n/2$, since otherwise we may delete all the at most $\varepsilon n^2/2$ edges touching $X \setminus S$ and thus make the graph empty (recall that $G_3[Y]$ is an empty graph). 
	
	\paragraph{Step 6.}
	% We have $|X \setminus S|,|Y| \geq \varepsilon n/2$. 
%	Note that $G_3[(X \setminus S) \cup Y]$ contains at most $N_3=M_{3}+N_{1}\leq 2\varepsilon_1 n^{3}$ induced copies of $D$. 
    Note that $G_3[(X \setminus S) \cup Y]$ contains at most $N + (K_1 + K_2 + K_3)n \leq 2\varepsilon_1 n^{3}$ induced copies of $D$.
    We have
	$$\frac{\varepsilon^2}{32} \cdot |X \setminus S| \cdot |Y| \cdot \min\{|X \setminus S|,|Y|\} \geq \frac{\varepsilon^2}{32} \left(\frac{\varepsilon n}{2}\right)^3 > 2\varepsilon_1 n^3,$$
	hence, we can apply Lemma \ref{lem:P2 independet_set} with parameter $\gamma = \varepsilon$ to make $G_3[(X \setminus S) \cup Y]$ induced $D$-free by changing at most $\varepsilon |X \setminus S||Y| \leq \varepsilon n^2/4=:K_6$ edges between $X \setminus S$ and $Y$. This makes the entire graph induced $D$-free. The overall number of edge changes in all steps is at most $$K_1+K_2+K_3+K_5+K_6<\varepsilon n^2,$$ contradicting the assumption that $G$ is $\varepsilon$-far from being induced $D$-free. This completes the proof.
\end{proof}

\section{Lower bounds}
In this section we prove the ``only if'' part of Theorem \ref{thm:induced}, as well as Theorem \ref{thm:noninduced}. Two subgraphs of a graph $G$ will be called {\em pair-disjoint} if they share at most one vertex. We will use the obvious fact that if $G$ contains a collection of $\varepsilon n^2$ pairwise pair-disjoint (induced) copies of $F$, then $G$ is $\varepsilon$-far from being (induced) $F$-free. 
We need the following simple claim. % For completeness, we include a proof.
%\begin{lemma}\label{lem:design}
%	For $k \geq 2$ and $r \geq 1$, there is a collection of at least $(r/k)^2$ $k$-tuples in $[r]^k$, any two of which identify on at most one coordinate. 
%\end{lemma}
%\begin{proof}
%	Construct the collection greedily: at each step, add one $k$-tuple $(x_1,\dots,x_k)$ to the collection and discard all $k$-tuples $(y_1,\dots,y_k)$ for which 
%	$\#\{1 \leq i \leq k : x_i = y_i\} \geq 2$. The number of $k$-tuples discarded at each step is at most $\binom{k}{2}r^{k-2}$. Hence, we get a collection of size at least
%	$$
%	\frac{r^k}{\binom{k}{2}r^{k-2} + 1} \geq (r/k)^2.
%	$$
%\end{proof}

\begin{lemma}\label{lem:design}
	For $k \geq 2$ and $r \geq 2k$, there is a collection $R \subseteq [r]^k$, $|R| \geq  r^2/4$, such that any two $k$-tuples in $R$ coincide on at most one coordinate. 
\end{lemma}

\begin{proof}
   Let $p$ be a prime such that $r/2< p\leq r$, which exists by Bertrand's postulate. For $a,b \in \mathbb{F}_p$, let $x_{a,b} \in \mathbb{F}_p^k$ be the $k$-tuple $x_{a,b}(i) = a + (i-1)b$, $i = 1,\dots,k$. 
	For $(a_1,b_1) \neq (a_2,b_2)$, there is at most one $1 \leq i \leq k$ with $x_{a_1,b_1}(i) = x_{a_2,b_2}(i)$. Indeed, if there are two such $1 \leq i \neq j \leq k$, then $a_1 + (i-1)b_1 = a_2 + (i-1)b_2$ and $a_1 + (j-1)b_1 = a_2 + (j-1)b_2$. Solving this system of equations gives $a_1 = a_2$ and $b_1 = b_2$, a contradiction. Here we use the fact that $i \not\equiv j \pmod{p}$, which follows from $p > r/2 \geq k$. 
	\end{proof}

We will adapt the constructions in \cite{Alon} and \cite{AS_induced} to ordered graphs. 
These constructions use generalizations of Behrend's example \cite{Behrend} of large sets of integers with no 3-term arithmetic progressions, see Lemma 3.1 in \cite{Alon} and Lemma 4.1 in \cite{AS_induced}. We will use the following common generalization of these two lemmas.
\begin{lemma}\label{lem:Behrend}
	For every $k \geq 3$ and $m$, there is $S \subseteq [m]$ of size at least $m \cdot e^{-c_k\sqrt{\log m}}$ such that for every $3 \leq t \leq k$ and for every choice of integers $p_1,\dots,p_{t-1} \geq 1$ with $p_1 + \dots + p_{t-1} \leq k$, 
	% there is no non-trivial solution to $p_1s_1 + \dots + p_{t-1}s_{t-1} = (p_1 + \dots + p_{t-1})s_t$ with $s_1,\dots,s_t \in S$.
	if $s_1,\dots,s_t \in S$ satisfy $p_1s_1 + \dots + p_{t-1}s_{t-1} = (p_1 + \dots + p_{t-1})s_t$ then $s_1 = \dots = s_t$.
\end{lemma}

The proof of Lemma \ref{lem:Behrend} is very similar to the proofs of the aforementioned lemmas from \cite{Alon,AS_induced} (which themselves closely follow Behrend's original argument). The proof is thus omitted.

The proof of the ``only if'' part of Theorem \ref{thm:induced} involves a case analysis over several small ordered graphs, each of which needs a slightly different variant of a construction from \cite{Alon}. To avoid repetitions, we now introduce a general setting in which this construction can be applied. 
A {\em pattern} $P$ is a complete ordered graph, say on $[k]$ with the natural ordering, whose edges are colored with the colors black, white and gray.
An ordered graph $G$ is said to have the pattern $P$ if there is a partition $V(G) = V_1 \cup \dots \cup V_k$ into independent sets with $V_1 < \dots < V_k$, such that the following condition is satisfied: for every $1 \leq i < j \leq k$, if $\{i,j\}$ is colored black then the bipartite graph $(V_i,V_j)$ is complete, and if $\{i,j\}$ is colored white then the bipartite graph $\{i,j\}$ is empty. The partition $(V_1,\dots,V_k)$ is called a {\em $P$-partition} of $G$.
A {\em completion} of $P$ is an ordered graph $F$ on $[k]$ which has pattern $P$, i.e. $\{i,j\} \in E(F)$ if $\{i,j\}$ is black in $P$ and $\{i,j\}\notin E(F)$ if $\{i,j\}$ is white in $P$. In other words, a completion is obtained by recoloring the gray edges with black/white. 
% Let $A \subseteq V(P) = V(F) = [k]$. We say that $F$ is {\em $(P,A)$-good} if for every ordered graph $G$ with $P$-partition $(V_1,\dots,V_k)$ and for every induced copy $F'$ of $F$ in $G$, there are vertices $v_i \in V_i \cap V(F')$, $i \in A$, such that $\{v_i : i \in A\}$ induces a copy of $F[A]$ in $G$.
Let $\mathcal{A}$ be a set of subsets of $V(P)$. We say that $F$ is {\em $(P,\mathcal{A})$-good} if for every ordered graph $G$ with $P$-partition $(V_1,\dots,V_k)$ and for every induced copy $F'$ of $F$ in $G$, there is $A \in \mathcal{A}$ and vertices $v_i \in V_i \cap V(F')$, $i \in A$, such that $(v_i : i \in A)$ form an (ordered) induced copy of $F[A]$. Note that in such a copy, $v_i$ must play the role of $i$ due to the vertex order.
If $\mathcal{A} = \{A\}$ then we will simply write ``$(P,A)$-good'' in place of ``$(P,\{A\})$-good''. 
The following is a generalization of the aforementioned construction from \cite{Alon}.

\begin{lemma}\label{lem:RS_cycle_general}
    Let $P$ be a pattern on $[k]$, let $\mathcal{A}$ be a set of subsets of $V(P)$, and suppose that there is a bijection $\sigma : [k] \rightarrow [k]$ such that for every $A \in \mathcal{A}$, there is a cycle $i_1,\dots,i_t,i_1$ of gray edges in $P[A]$ such that $\sigma(i_1) < \dots < \sigma(i_t)$. Let $F$ be a completion of $P$ which is $(P,\mathcal{A})$-good. Then for every small enough $\varepsilon$ and $n \geq n_0(\varepsilon)$, there is an $n$-vertex ordered graph $G$ with the following properties:
    \begin{enumerate}
    \item $G$ has pattern $P$. 
    \item $G$ contains $\varepsilon n^2$ pair-disjoint induced copies of $F$.
	\item $G$ contains at most $\varepsilon^{\Omega(\log 1/\varepsilon)}n^{v(F)}$ induced copies of $F$. 
	\end{enumerate}
\end{lemma}

\begin{proof}
As above, we assume that $V(F) = V(P) = [k]$
(with the natural vertex order). Let $m$ be the maximal integer satisfying $e^{-c_k \sqrt{\log m}} \geq 4k^4\varepsilon$, where $c_k$ is from Lemma \ref{lem:Behrend}. It is easy to see that $m \geq (1/\varepsilon)^{\Omega(\log 1/\varepsilon)}$.
Let $S$ be the set guaranteed by Lemma \ref{lem:Behrend}; so $|S| \geq 4k^4\varepsilon m$ by our choice of $m$.
Let $\sigma : [k] \rightarrow [k]$ be a bijection as in the statement of the lemma. We start by defining an ordered graph $H$, as follows. The vertex-set of $H$ consists of $k$ pairwise-disjoint independent sets $V_1,\dots,V_k$ with $V_1 < \dots < V_k$. For each $1 \leq i \leq k$, the set $V_i$ is identified with $[\sigma(i) \cdot m]$. 
So $|V(H)| = \sum_{i=1}^{k}{i \cdot m} \leq k^2m$. 
For each $x \in [m]$ and $s \in S$, add a copy of $F$ on vertices $v_1,\dots,v_k$, where $v_i = x + (\sigma(i)-1) \cdot s \in V_i$; this copy is denoted by $F_{x,s}$.
Next, for each black edge $\{i,j\}$ of $P$, make the bipartite graph $(V_i,V_j)$ complete, and for each white edge $\{i,j\}$ of $P$, make the bipartite graph $(V_i,V_j)$ empty. This agrees with the copies $F_{x,s}$, since $F$ has pattern $P$. 
Finally, for each gray edge $\{i,j\}$ of $P$ and for each $v_i \in V_i, v_j \in V_j$, if $\{v_i,v_j\}$ is not contained in any of the copies $F_{x,s}$, then make $\{v_i,v_j\}$ an edge of $H$ if $\{i,j\} \notin E(F)$, and a non-edge of $H$ if $\{i,j\} \in E(F)$.
By construction, $H$ has pattern $P$ with $P$-partition $(V_1,\dots,V_k)$.

For distinct pairs $(x_1,s_1),(x_2,s_2) \in [m] \times S$, the copies $F_{x_1,s_1}$ and $F_{x_2,s_2}$ are pair-disjoint. Indeed, if $F_{x_1,s_1}$ and $F_{x_2,s_2}$ have the same vertex in $V_i$ and $V_j$ (for some $1 \leq i < j \leq k$), then $x_1 + (\sigma(i)-1)s_1 = x_2 + (\sigma(i)-1)s_2$ and $x_1 + (\sigma(j)-1)s_1 = x_2 + (\sigma(j)-1)s_2$. Solving this system of equations, we get that $x_1 = x_2$ and $s_1 = s_2$. So we conclude that the copies $(F_{x,s})_{(x,s) \in [m] \times S}$ of $F$ are pair-disjoint and hence induced. The number of these copies is $m|S| \geq 4k^4 \varepsilon m^2$.

Now, let $G$ be the $\frac{n}{v(H)}$-blowup of $H$. For $1 \leq i \leq k$, denote by $W_i$ the blowup of $V_i$. 
It is easy to see that $G$ has pattern $P$ with $P$-partition $(W_1,\dots,W_k)$. 
Each induced copy of $F$ in $H$ gives rise to $\big(\frac{n}{2v(H)}\big)^2$ pair-disjoint induced copies of $F$ in $G$, by Lemma \ref{lem:design} with parameter $r=\frac{n}{v(H)}$. Hence, $G$ contains a collection of 
$4k^4\varepsilon m^2 \cdot \big(\frac{n}{2v(H)}\big)^2 \geq \varepsilon n^2$ pair-disjoint induced copies of $F$.

To complete the proof it remains to show that item 3 holds. So let $F'$ be an induced copy of $F$ in $G$. 
Since $F$ is $(P,\mathcal{A})$-good, there is $A \in \mathcal{A}$ and vertices $w_i \in W_i \cap V(F')$ for $i \in A$, such that
$(w_i : i \in A)$ form an induced copy of $F[A]$ (in $G$). 
For $i \in A$, let $v_i \in V_i$ be such that $w_i$ belongs to the blowup of $v_i$. Then $(v_i : i \in A)$ form an induced copy of $F[A]$ in $H$. 
By the assumption of the lemma, there is a cycle $i_1,\dots,i_t,i_1$ in $P[A]$ consisting of gray edges, such that $\sigma(i_1) < \dots < \sigma(i_t)$. 
By the construction of $H$, for every gray edge $\{i,j\}$ of $P[A]$, it must be that $\{v_i,v_j\}$ is contained in $V(F_{x,s})$ for some $x \in [m], s \in S$. Indeed, if $\{v_i,v_j\}$ is not contained in any $F_{x,s}$, then the adjacency relation of $\{v_i,v_j\}$ is opposite to the adjacency relation of $\{i,j\}$ in $F$. 
So we see that for every $1 \leq j \leq t$, there are $x_j \in X, s_j \in S$ such that $\{v_{i_j},v_{i_{j+1}}\} \subseteq V(F_{x_j,s_j})$ (with indices taken modulo $t$). For $1 \leq j \leq t-1$, this means that $v_{i_j} = x_j + (\sigma(i_j) - 1)s_j$ and $v_{i_{j+1}} = x_j + (\sigma(i_{j+1}) - 1)s_j$; hence, 
$v_{i_{j+1}} - v_{i_j} = 
(\sigma(i_{j+1}) - \sigma(i_j)) \cdot s_j$. Similarly, for $j = t$ we have $v_{i_1} = x_t + (\sigma(i_1) - 1) \cdot s_t$ and $v_{i_t} = x_t + (\sigma(i_t) - 1) \cdot s_t$, and hence $v_{i_t} - v_{i_1} = (\sigma(i_{t}) - \sigma(i_1)) \cdot s_t$. So we get that 
$(\sigma(i_2)-\sigma(i_1)) \cdot s_1 + \dots + (\sigma(i_t) - \sigma(i_{t-1})) \cdot s_{t-1} = (\sigma(i_t) - \sigma(i_1)) \cdot s_t$. 
Now we use our choice of $S$ via Lemma \ref{lem:Behrend}, taking $p_1,\dots,p_{t-1}$ in Lemma \ref{lem:Behrend} to be $p_j = \sigma(i_{j+1}) - \sigma(i_j)$. (Here we use that $\sigma(i_1) < \dots < \sigma(i_t)$ so that $p_1,\dots,p_{t-1}$ are positive.) We obtain that $s_1 = \dots = s_t =: s$. 
We now get that $x_j = v_{i_{j+1}} - (\sigma(i_{j+1}) - 1) \cdot s = x_{j+1}$ (for every $1 \leq j \leq t-1$), and hence $x_1 = \dots = x_t$. 

So far we have shown that for every induced copy $F'$ of $F$ in $G$, 
there exist $x \in [m]$, $s \in S$, a set $A \in \mathcal{A}$, a cycle $i_1,\dots,i_t,i_1$ in $P[A]$,
and vertices
$v_{i_j} \in V_{i_j}$ and $w_{i_j} \in W_{i_j} \cap V(F')$, such that $w_{i_j}$ belongs to the blowup of $v_{i_j}$, and $v_{i_1},\dots,v_{i_t} \in V(F_{x,s})$. There are $|\mathcal{A}| \leq 2^k = O(1)$ choices for $A$, and fixing $A$ determines $i_1,\dots,i_t$.
The number of choices for $(x,s)$ is $m|S| \leq m^2$, and fixing $x,s$ determines
$v_{i_1},\dots,v_{i_t}$. Now, given $v_{i_1},\dots,v_{i_t}$, there are 
$\big( \frac{n}{v(H)} \big)^t$ choices for $w_{i_1},\dots,w_{i_t}$, and at most $n^{k-t}$ choices for the remaining $k - t$ vertices of $F$. 
Hence, given $v_{i_1},\dots,v_{i_t}$, the number of choices for an induced copy of $F$ is at most $\big( \frac{n}{v(H)} \big)^t \cdot n^{k - t} \leq \big( \frac{n}{v(H)} \big)^3 \cdot n^{k - 3} = n^k/v(H)^3 \leq n^k/m^3$.
So overall, the number of induced copies of $F$ in $G$ is at most
$O(1) \cdot m^2 \cdot n^{k}/m^3 = O(n^k/m) \leq \varepsilon^{\Omega(\log 1/\varepsilon)} n^{k}$, as required. 
\end{proof}

Evidently, for every specific cycle $i_1,\dots,i_t,i_1$ in $P$, one can choose a bijection $\sigma : [k] \rightarrow [k]$ with $\sigma(i_1) < \dots < \sigma(i_t)$. Hence, a bijection $\sigma$ as in Lemma \ref{lem:RS_cycle_general} always exists when $\mathcal{A} = \{A\}$. We therefore have the following corollary:

\begin{lemma}\label{lem:RS_cycle}
    Let $P$ be a pattern, and let $A \subseteq V(P)$ such that $P[A]$ has a cycle consisting of gray edges. Let $F$ be a $(P,A)$-good completion of $P$. Then the conclusion of Lemma \ref{lem:RS_cycle_general} holds. 
\end{lemma}

\noindent
Lemma \ref{lem:RS_cycle_general} implies the following statement, which extends a construction from \cite{AS_induced} to ordered graphs. 

\begin{lemma}\label{lem:RS_triangles}
	Let $F$ be an ordered graph which contains a triangle. Then for every sufficiently small $\varepsilon>0$ and $n>n_0(\varepsilon)$, there is an $n$-vertex ordered graph $G$ which contains $\varepsilon n^2$ pair-disjoint induced copies of $F$, but only $\varepsilon^{\Omega(\log 1/\varepsilon)}n^{v(F)}$ induced copies of $F$ altogether. 
\end{lemma}
\begin{proof}
    Take $P$ to be the pattern on $V(F)$ in which all non-edges of $F$ are white and all edges of $F$ are gray. Take $\mathcal{A}$ to be the set of all (vertex sets of) triangles in $F$.
    Observe that $F$ is $(P,\mathcal{A})$-good. Indeed, let $G$ be an ordered graph with $P$-partition $(V_1,\dots,V_k)$ and let $F'$ be a copy of $F$ in $G$. Then $F'$ has a triangle, say on vertices $v_i \in V_i, v_j \in V_j, v_{\ell} \in V_{\ell}$. Then $\{i,j\},\{i,\ell\},\{j,\ell\}$ must be gray edges in $P$ and hence must be edges in $F$. So $i,j,\ell$ is a triangle in $F$, meaning that $v_i,v_j,v_{\ell}$ form a copy of $F[A]$ for $A = \{i,j,\ell\} \in \mathcal{A}$.
    Now take $\sigma : [k] \rightarrow [k]$ to be an arbitrary bijection. For a triangle it is always possible to choose a starting point and an orientation such that the triangle (as a cycle) is increasing with respect to $\sigma$.
    So the conditions of Lemma \ref{lem:RS_cycle_general} are satisfied, and the assertion follows from Lemma \ref{lem:RS_cycle_general}. 
\end{proof}

\noindent
Lemma \ref{lem:RS_cycle} easily implies the following:

\begin{lemma}\label{lem:RS_core}
	Let $K$ be an ordered core which has a cycle. Then for every small enough $\varepsilon$ and large enough $n$, there is an $n$-vertex ordered graph $G$ with the following properties:
	\begin{enumerate}
		\item $G$ is homomorphic to $K$.
		\item $G$ contains $\varepsilon n^2$ pair-disjoint induced copies of $K$.
		\item $G$ has at most $\varepsilon^{\Omega(\log 1/\varepsilon)}n^{v(K)}$ (not necessarily induced) copies of $K$. 
	\end{enumerate}
\end{lemma}
\begin{proof}
	We reduce to Lemma \ref{lem:RS_cycle}. 
	Let $P$ be the pattern on $[k] = V(K)$ in which $\{i,j\}$ is gray if $\{i,j\} \in E(K)$ and white if $\{i,j\} \notin E(K)$. Observe that an ordered graph $G$ has pattern $P$ if and only if $G$ is homomorphic to $K$. The fact that $K$ is a core implies that $K$ is $(P,V(P))$-good. 
    Apply Lemma \ref{lem:RS_cycle} to get an ordered graph $G$ satisfying items 1-3 in Lemma \ref{lem:RS_cycle_general}. Then $G$ is homomorphic to $K$ because $G$ has pattern $P$. Observe that every copy of $K$ in $G$ is induced because $K$ is a core and $G$ is homomorphic to $K$. 
    Lemma \ref{lem:RS_core} follows. 
%     Note also that if $G$ is homomorphic to $K$, then every copy of $K$ in $G$ must be induced (because $K$ is a core). 
% 	The assertion now follows from Lemma \ref{lem:RS_cycle}.
\end{proof}
\noindent
Using Lemma \ref{lem:RS_core}, we can prove Theorem \ref{thm:noninduced}.
\begin{proof}[Proof of Theorem \ref{thm:noninduced}]
    Let $K$ be the core of $F$, and suppose that $V(K) = [k]$.
	Apply Lemma \ref{lem:RS_core} with parameters $v(F)^2 \cdot \varepsilon$ (in place of $\varepsilon$) and $\frac{n}{v(F)}$ (in place of $n$) to obtain an ordered graph $G'$ on $\frac{n}{v(F)}$ vertices with the properties stated in the lemma. 
	Since $G'$ is homomorphic to $K$, we have $V(G') = V_1 \cup \dots \cup V_k$ for independent sets $V_1 < \dots < V_k$.  
	Let $G$ be the $v(F)$-blowup of $G'$, and denote by $W_i$ the blowup of $V_i$ 
	(for $i = 1,\dots,k$).
	By item 2 in Lemma \ref{lem:RS_core}, $G'$ contains a collection $K_1,\dots,K_M$ of $M \geq v(F)^2 \cdot \varepsilon \cdot (\frac{n}{v(F)})^2 = \varepsilon n^2$ pair-disjoint copies of $K$. For each $1 \leq i \leq M$, let $B_i$ be the $v(F)$-blowup of $K_i$, and let $E_i$ be the set of edges of $B_i$ which go between the sets $W_1,\dots,W_k$. Since $K_1,\dots,K_M$ are pair-disjoint, the sets $E_1,\dots,E_M$ are disjoint. Observe that each $B_i$ contains a copy of $F$, and that in order to destroy this copy one must delete some edge of $E_i$. Since $E_1,\dots,E_M$ are disjoint, one must delete at least $M \geq \varepsilon n^2$ edges from $G$ to make it $F$-free, as required. 
	
	To complete the proof, let us bound the number of copies of $F$ in $G$.
	Since $K$ is a subgraph of $F$, every copy of $F$ must contain a copy of $K$. Each copy of $K$ can be completed to a copy of $F$ in at most $n^{v(F)-k}$ ways. 
	Since $K$ is a core and $G'$ (and hence also $G$) is homomorphic to $K$, every copy of $K$ in $G$ corresponds to a copy of $K$ in $G'$. On the other hand, each copy of $K$ in $G'$ gives rise to $v(F)^k = O(1)$ copies of $K$ in $G$. By item 3 in Lemma \ref{lem:RS_core}, $G'$ has at most $\varepsilon^{\Omega(\log 1/\varepsilon)} \cdot n^k$ copies of $K$. 
	Hence, the number of copies of $F$ in $G$ is at most $O(1) \cdot \varepsilon^{\Omega(\log 1/\varepsilon)} \cdot n^k \cdot n^{v(F) - k} = \varepsilon^{\Omega(\log 1/\varepsilon)} n^{v(F)}$, as required.  
\end{proof}

\noindent
Next, we prove the ``only if'' part of Theorem \ref{thm:induced}, which we restate as follows:

\begin{theorem}\label{thm:construction_induced}
	Let $F$ be an ordered graph, $v(F)\geq 3$,
	$F \notin \left\{D,D^{\leftarrow},\overline{D},\overline{D^{\leftarrow}}\right\}$. Then for every sufficiently small $\varepsilon>0$ and $n>n_0(\varepsilon)$, there is an $n$-vertex ordered graph $G$ which is $\varepsilon$-far from being induced $F$-free but contains at most $\varepsilon^{\Omega(\log 1/\varepsilon)} n^{v(F)}$ induced copies of $F$.
\end{theorem}
\begin{proof}
	If $F$ contains a triangle then the assertion follows from Lemma \ref{lem:RS_triangles}. By symmetry with respect to taking graph complements, the same is true if $F$ contains an independent set of size 3. This in particular proves the theorem for $F$ on at least $6$ vertices, since every such $F$ contains a triangle or an independent set of size 3. We will assume that $F$ contains neither of these.
	
	Denote by $P_k^{\text{mon}}$ the monotone path with $k$ vertices, that is, the ordered path with vertex set $[k]$ and edges $\{i,i+1\}$ for $i=1,\dots,k-1$.
	It is easy to see that $P_k^{\text{mon}}$ is $(P,V(P))$-good for the pattern $P$ on $[k]$ in which $\{1,2\},\{2,3\},\dots,\{k-1,k\},\{k,1\}$ are gray and all other edges are white. So if $F = P_k^{\text{mon}}$ ($k \geq 3$) then the assertion of the theorem follows from Lemma \ref{lem:RS_cycle}.  
	
	If $v(F) = 3$, then $F = P_{3}^{\text{mon}}$ or $F = \overline{P_{3}^{\text{mon}}}$, and we already handled these cases. 
	If $v(F) = 5$ then $F$ must be a $5$-cycle, because every other 5-vertex graph contains a triangle or an independent set of size 3.
	Every ordered $5$-cycle is a core (because the homomorphic image of an odd cycle must itself contain an odd cycle). So for these $F$, the assertion follows from Lemma \ref{lem:RS_core}.
	
	It remains to handle the case $v(F) = 4$. The only 4-vertex (unordered) graphs which have no triangle and no independent set of size $3$ are the 4-cycle, the complement of the 4-cycle, and the path with four vertices.

	We first consider the 4-cycle. There are 3 non-isomorphic ordered 4-cycles. Assuming the vertices are $1,2,3,4$, these 4-cycles are:
	$C_4^{(1)} = 1,2,3,4,1$; 
	$C_4^{(2)} = 1,3,2,4,1$;
	and
	$C_4^{(3)} = 1,2,4,3,1$. See Figure \ref{fig:4cycles} for an illustration. 
	
	\begin{itemize}
	    \item[$C_4^{(1)}$:] It is easy to see that $C_4^{(1)}$ is a core, so this case follows from Lemma \ref{lem:RS_core}. 
	    \item[$C_4^{(2)}$:]  Let us consider the complement $\overline{C_4^{(2)}}$, which is the ordered graph with vertices $1,2,3,4$ and edges $\{1,2\},\{3,4\}$. We show that this graph is $P$-good for a suitable pattern $P$. Let $P$ be the pattern on $[4]$ in which $\{1,2\},\{3,4\}$ are black and all other edges are gray. Then $\overline{C_4^{(2)}}$ is $(P,V(P))$-good. Indeed, let $G$ be a graph with pattern $P$, and let $(V_1,V_2,V_3,V_4)$ be a $P$-partition of $G$. Note that the bipartite graphs $(V_1,V_2)$ and $(V_3,V_4)$ are complete. Let $C$ be an induced copy of $\overline{C_4^{(2)}}$ in $G$, and let $a_i$ be the vertex of $C$ playing the role of $i$ (for $i = 1,\dots,4$). It is enough to show that $|C \cap V_i| = 1$ for all $1 \leq i \leq 4$, as this would imply that $a_i \in V_i$. Suppose by contradiction that $|C \cap V_i| \geq 2$ for some $i$. Then $|C \cap V_i| = 2$ because $V_i$ is an independent set and $C$ is does not have an independent set of size 3. The two vertices in $C \cap V_i$ must play the role of some non-edge $e$ of $C$. If $e$ is  $\{a_1,a_3\},\{a_2,a_4\}$ or $\{a_1,a_4\}$, then $|C \cap V_i| \geq 3$, because for each of those edges, there is another vertex of $C$ between the endpoints of the edge. Hence $e = \{a_2,a_3\}$. Since $a_1$ comes before $V_i$ and $a_4$ after $V_i$, it must be that $i = 2$ or $i = 3$; without loss of generality, $i = 2$. Then $a_1 \in V_1$. But then $a_1$ is adjacent to $a_3 \in V_2$ because the bipartite graph $(V_1,V_2)$ is complete, a contradiction. 
	    
	    The assertion of Theorem \ref{thm:construction_induced} for $C_4^{(2)}$ now follows by applying Lemma \ref{lem:RS_cycle} to $\overline{C_4^{(2)}}$ and taking complements.
	    
	    \item[$C_4^{(3)}$:] 	Here we will use Lemma \ref{lem:RS_cycle_general}.
	Let $P$ be the pattern on $[4]$ with white edge $\{2,3\}$ and all other edges gray. Let $\mathcal{A} = \{\{1,2,4\},\{1,3,4\}\}$. 
	We claim that $C_4^{(3)}$ is $(P,\mathcal{A})$-good. Indeed, let $G$ be a graph with pattern $P$, and let $(V_1,V_2,V_3,V_4)$ be a $P$-partition of $G$. Note that the bipartite graph $(V_2,V_3)$ is empty. Let $C$ be an induced copy of $C_4^{(3)}$ in $G$, and let $a_i$ be the vertex of $C$ playing the role of $i$ (for $i = 1,\dots,4$). We need to show that $a_1 \in V_1$, $a_4 \in V_4$, and $a_2 \in V_2$ or $a_3 \in V_3$. Observe that $a_2,a_3 \in V_2 \cup V_3$, because $a_1,a_4$ are adjacent to both $a_2$ and $a_3$, $a_1$ comes before $a_2,a_3$, and $a_4$ comes after $a_2,a_3$. If $a_2 \in V_2$ and $a_3 \in V_3$ then we must have $a_1 \in V_1$ and $a_4 \in V_4$, so we are done. Suppose then that $a_2,a_3 \in V_2$ or $a_2,a_3 \in V_3$; without loss of generality, we may assume that $a_2,a_3 \in V_2$. This implies that $a_1 \in V_1$. Also, since the bipartite graph $(V_2,V_3)$ is empty, we must have $a_4 \in V_4$. Therefore, $C_4^{(3)}$ is $(P,\mathcal{A})$-good. So the assertion of Theorem \ref{thm:construction_induced} for $C_4^{(3)}$ follows from Lemma \ref{lem:RS_cycle_general} (with $\sigma$ being the identity map). 
	\end{itemize}
	
	\begin{figure}\label{fig:4cycles}
	\centering
	\begin{tikzpicture}
		\foreach \i in {1,2,3,4}
		{
			\draw (\i,0) node[fill=black,circle,minimum size=2pt,inner sep=0pt] {};
			\draw (\i, -0.25) node {$\i$};
		
			\draw (1,0) -- (2,0) -- (3,0) -- (4,0);
			\draw[rounded corners = 12] (4,0) -- (3,0.4) -- (2,0.4) -- (1,0);
		}
		\draw (2.5,-0.75) node {$C_4^{(1)}$};
	\end{tikzpicture}
	\hspace{0.2cm}
	\begin{tikzpicture}
	\foreach \i in {1,2,3,4}
	{
		\draw (\i,0) node[fill=black,circle,minimum size=2pt,inner sep=0pt] {};
		\draw (\i, -0.25) node {$\i$};
		
		\draw (2,0) -- (3,0);
		\draw[rounded corners = 12] (1,0) -- (2,0.3) -- (3,0);
		\draw[rounded corners = 12] (2,0) -- (3,0.3) -- (4,0);
		\draw[rounded corners = 12] (4,0) -- (3,0.4) -- (2,0.4) -- (1,0);
	}
	\draw (2.5,-0.75) node {$C_4^{(2)}$};
	\end{tikzpicture}
	\hspace{0.2cm}
	\begin{tikzpicture}
	\foreach \i in {1,2,3,4}
	{
		\draw (\i,0) node[fill=black,circle,minimum size=2pt,inner sep=0pt] {};
		\draw (\i, -0.25) node {$\i$};
		
		\draw (1,0) -- (2,0);
		\draw (3,0) -- (4,0);
		\draw[rounded corners = 12] (1,0) -- (2,0.4) -- (3,0);
		\draw[rounded corners = 12] (2,0) -- (3,0.4) -- (4,0);
	}
	\draw (2.5,-0.75) node {$C_4^{(3)}$};
	\end{tikzpicture}
\caption{The ordered $4$-cycles}
\end{figure}
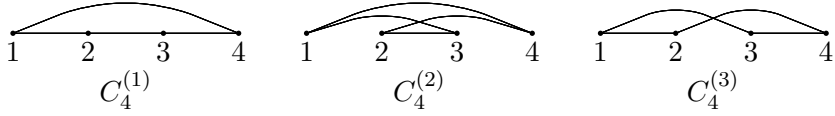

	It remains to consider ordered paths with four vertices. Up to complementation and order reversal, there are only four possible such paths: $P_4^{\text{mon}}$; 
	$P_4^{(1)} = 2,1,4,3$; 
	$P_4^{(2)} = 2,1,3,4$; and
	$P_4^{(3)} = 3,2,1,4$.
	See Figure \ref{fig:paths} for an illustration. We already established the case $P_4^{\text{mon}}$. 
	For the other three cases, we again use Lemma \ref{lem:RS_cycle}. 
	
	\begin{itemize}
	    \item[$P_4^{(1)}$:] 	Let $P$ be the pattern on $[4]$ in which $\{1,3\}$ and $\{2,4\}$ are white, and all other edges are gray. Then $P_4^{(1)}$ is $(P,V(P))$-good. Indeed, let $G$ be a graph with pattern $P$, and let $(V_1,\dots,V_4)$ be a $P$-partition of $G$. Then the bipartite graphs $(V_1,V_3)$ and $(V_2,V_4)$ are empty.
	Let $X$ be an induced copy of $P_4^{(1)}$ in $G$, and let $a_i$ be the vertex of $X$ playing the role of $i$ (for $i = 1,\dots,4$). 
% 	Let us show that $a_i$
	It is enough to show that 
	$|X \cap V_i| = 1$ for all $1 \leq i \leq 4$, as this will imply that $a_i \in V_i$, showing that $P_4^{(1)}$ is $(P,V(P))$-good.
	Suppose by contradiction that 
	$|X \cap V_i| \geq 2$ for some $i$. Then $|X \cap V_i| = 2$ because $V_i$ is an independent set and $X$ has no independent set of size $3$. The two vertices in 
	$X \cap V_i$ must play the role of some non-edge $e$ of $X$. If 
	$e = \{a_1,a_3\},\{a_2,a_4\}$ then 
	$|X \cap V_i| \geq 3$, because for each of those $e$, there is another vertex of $X$ between the endpoints of $e$. Hence 
	$e = \{a_2,a_3\}$. Since $a_1$ comes before $V_i$ and $a_4$ after $V_i$, it must be that $i = 2$ or $i = 3$; without loss of generality, $i = 2$. Then $a_1 \in V_1$. It follows that $a_4 \in V_4$, because $a_4$ is adjacent to $a_1$ and there are no edges between $V_1$ and $V_3$. But now, $a_4$ is not adjacent to $a_3 \in V_2$, as there are no edges between $V_2$ and $V_4$. This is a contradiction.
	
	Lemma \ref{lem:RS_cycle} now confirms the case of $P_4^{(1)}$. 
	
	\item[$P_4^{(2)}$:] Take $P$ to be the pattern on $[4]$ with edges $\{2,3\},\{2,4\}$ white and all other edges gray, and take $A = \{1,3,4\}$. Then $P_4^{(2)}$ is $(P,A)$-good. Indeed, let $G$ be an ordered graph with pattern $P$, and let $(V_1,V_2,V_3,V_4)$ be a $P$-partition of $G$. Note that the bipartite graphs $(V_2,V_3)$ and $(V_2,V_4)$ are empty. Let $X$ be an induced copy of $P_4^{(2)}$ in $G$, and let $a_i$ be the vertex of $X$ playing the role of $i$ (for $i = 1,\dots,4$). We need to show that $a_i \in V_i$ for $i = 1,3,4$. 
	We first claim that $a_3 \in V_3$. Since $a_1,a_4$ are adjacent to $a_3$, $a_1$ comes before $a_3$, and $a_4$ comes after $a_3$, it must be that $a_3 \in V_2 \cup V_3$. If $a_3 \in V_2$ then $a_4$ cannot be adjacent to $a_3$ because the bipartite graphs $(V_2,V_3)$ and $(V_2,V_4)$ are empty, a contradiction. So $a_3 \in V_3$. It follows that $a_4 \in V_4$. Similarly, as the bipartite graph $(V_2,V_3)$ is empty, it must be that $a_1 \in V_1$, as required.
	
	We apply Lemma \ref{lem:RS_cycle} to conclude.
	
	\item [$P_4^{(3)}$:] Take $P$ to be the pattern on $[4]$ with edges $\{2,4\},\{3,4\}$ white and all other edges gray, and take $A = \{1,2,3\}$. Similarly as in the previous case, one can check that $P_4^{(3)}$ is $(P,A)$-good. Therefore, we can again apply Lemma \ref{lem:RS_cycle} to finish the proof.
	\end{itemize}

\end{proof}

\begin{figure}\label{fig:paths}
	\centering
	\begin{tikzpicture}
	\foreach \i in {1,2,3,4}
	{
		\draw (\i,0) node[fill=black,circle,minimum size=2pt,inner sep=0pt] {};
		\draw (\i, -0.25) node {$\i$};
		
		\draw (1,0) -- (2,0);
		\draw (3,0) -- (4,0);
		\draw[rounded corners = 12] (4,0) -- (3,0.4) -- (2,0.4) -- (1,0);
	}
	\draw (2.5,-0.75) node {$P_4^{(1)}$};
	\end{tikzpicture}
	\hspace{0.2cm}
	\begin{tikzpicture}
	\foreach \i in {1,2,3,4}
	{
		\draw (\i,0) node[fill=black,circle,minimum size=2pt,inner sep=0pt] {};
		\draw (\i, -0.25) node {$\i$};
		
		\draw (1,0) -- (2,0);
		\draw (3,0) -- (4,0);
		\draw[rounded corners = 12] (1,0) -- (2,0.4) -- (3,0);
	}
	\draw (2.5,-0.75) node {$P_4^{(2)}$};
	\end{tikzpicture}
	\hspace{0.2cm}
	\begin{tikzpicture}
	\foreach \i in {1,2,3,4}
	{
		\draw (\i,0) node[fill=black,circle,minimum size=2pt,inner sep=0pt] {};
		\draw (\i, -0.25) node {$\i$};
		
		\draw (1,0) -- (2,0);
		\draw (2,0) -- (3,0);
		\draw[rounded corners = 12] (4,0) -- (3,0.4) -- (2,0.4) -- (1,0);
	}
	\draw (2.5,-0.75) node {$P_4^{(3)}$};
	\end{tikzpicture}
\caption{The ordered paths $P_4^{(i)}$}
\end{figure}
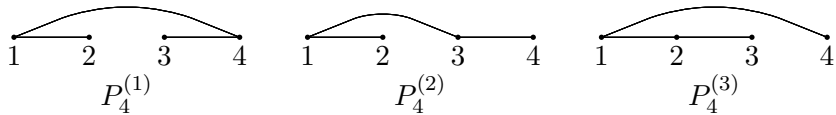

 	\paragraph{Acknowledgements} The authors are grateful to Omri Ben-Eliezer for telling them about the paper \cite{RR}, and the anonymous referee for their useful comments and suggestions.
 	
\bibliographystyle{alphaurl}
\bibliography{biblio.bib}
\end{document}